\documentclass{IOS-Book-Article}     

\usepackage{amssymb, hyperref}
 
\usepackage{amsthm, amsmath,  amssymb, amsbsy,   amsfonts, latexsym,  amsopn,   amstext, amsxtra,   euscript,   amscd}

\usepackage{graphicx}

\usepackage{cite}

\hypersetup{
  colorlinks   = true, 
  urlcolor     = blue, 
  linkcolor    = blue, 
  citecolor   = red 
}

\usepackage[msc-links]{amsrefs}

\DeclareMathOperator\re{Re }
\DeclareMathOperator\im{Im }

\newcommand\Z{\mathbb Z}
\newcommand\Q{\mathbb Q}
\newcommand\R{\mathbb R}
\newcommand\C{\mathbb C}

\newcommand\J{\mathcal J}
\def\P{\mathbb P}
\def\H{\mathcal H}

\newcommand\F{\mathcal F}
\def\O{\mathfrak O}

\newcommand\z{\xi}

\def\bx{\textbf{x}}
\def\by{\textbf{y}}

\newcommand\D{\Delta}

\newcommand\<{\langle}
\def\>{\rangle}

\def\a{\alpha}
\def\b{\beta}
\def\g{\gamma}
\def\d{{\delta }}

\newtheorem{thm}{Theorem}

\newtheorem {prop}{Proposition}
\newtheorem {lem}{Lemma}
\newtheorem {defn}{Definition}
\newtheorem {exa}{Example}

\newtheorem{alg}{Algorithm}
\newtheorem {rem}{Remark}
\newtheorem {cor}{Corollary}

\begin{document}
\begin{frontmatter}          
%
\title{Reduction theory of binary forms}
\runningtitle{}

 
%

\author{\fnms{Lubjana} \snm{Beshaj}} 


\address{Department of Mathematics, \\ Oakland University, \\ Rochester, MI, USA;\\  E-mail: beshaj@oakland.edu}

\begin{abstract}
In these lectures we give an introduction to the reduction theory of binary forms starting with quadratic forms with real coefficients, Hermitian forms, and then define the  Julia quadratic for any degree $n$ binary form. A survey of a reduction algorithm over $\mathbb Z$ is described based on recent work of Cremona and Stoll. 
\end{abstract}


\end{frontmatter}

 

\section*{Introduction}

The goal of these lectures is to give  an introduction to reduction theory of binary forms.   Since Gauss, the reduction theory of integral binary quadratic forms is quite completely understood. For binary Hermitian forms, this was studied starting with Hermite, Bianchi, and much developed by Elstrodt, Grunewland, and Mennicke  see \cite{egm}. 

In 1848, Hermite introduced a reduction theory for binary forms of degree $n$ which was developed more fully in 1917 by Julia in his thesis. For reducing binary forms of degree $n$, Julia introduced an irrational $SL_2(\Z)$-invariant of binary forms which is known in the literature as Julia's invariant. 

More recent work on this subject is done by Cremona in \cite{cremona-red}, where he gives reduction theory of binary cubic and quartic forms, and then Stoll and Cremona in \cite{stoll-cremona} for binary forms of degree $n \geq 2$.

We use reduction theory of  binary forms to study   the  following   problems. 

1)  For any  binary form $f$ defined over a  number  field $K$, find an  $GL_2(K)$-equivalent one $f^\prime$ such that  $f^\prime$ has minimal height as defined in \cite{nato-8}.

2)  Given a fixed value of the discriminant $\D$,  or a given set of invariants,  enumerate up to an $SL_2(\O_K)$-equivalence, all forms $f$ with the given discriminant or this set of invariants.

In Part 1, we start with   the basic theory of quadratic forms, see \cite{lam}.  In these lectures we will consider only positive definite binary quadratic forms, for  negative definite and indefinite see \cite{binary-quadratic}. Then, we give  a brief description of the modular action on the upper half plane, the fundamental domain, and the zero map which is a one to one map from the set of positive definite binary forms to the complex upper half plane $\H_2$.  

A positive definite binary quadratic form is called reduced when its image under the zero map is in the fundamental domain of the action of the modular group on the upper half plane.  Hence, the basic principle behind reduction theory is to associate to any positive definite quadratic a covariant point   in the complex upper half plane.  

The concluding section of Part 1 describes a method of reducing a   positive definite quadratic form to a reduced form and  also  an algorithm to count the number of reduced forms with fixed discriminant $\D$.  The method for reducing positive definite binary quadratics described in Section \ref{reduction-quadratic}  will be used in Part 3 to provide a reduction algorithm for any degree $n$ binary form defined over $\R$. 

To define reduction theory of degree $n$ binary forms defined over $\C$ first we have to consider reduction of binary quadratic Hermitian forms. 

 In Part 2,  Section \ref{hermitian-forms} we give some preliminaries  about Hermitian forms and Hermitian matrices,  see \cite{egm}, then we define the hyperbolic upper half space $\H_3$, and the zero map which gives a one to one correspondence between positive definite binary Hermitian forms and points in $\H_3$. We describe in detail the action of the special linear group $SL_2(\C)$ on the set of binary quadratic Hermitian forms as well as on $\H_3$.  But to define reduction theory for binary forms with complex coefficients we need  a discrete subring of $\C$ and a description of the fundamental domain.  

In section \ref{reduction-hermitian},  we consider the case when  $K$ is an imaginary quadratic number field, $K=\Q(\sqrt D)$, $D$ is a negative square free integer, and $\O_K$ it's ring of integers.  The "Bianchi group"  $\Gamma = SL_2( \O_K)$ acts on positive definite binary quadratic Hermitian forms preserving discriminants, and also (discretely) on $\H_3$. The latter action has a fundamental region $\F_K$, depending on $K$, as described in Section \ref{reduction-hermitian}.   We define  a positive definite binary quadratic Hermitian form to be reduced in the same way as we did in part 1. 

As an example we consider the case  when $K=\Q(i)$ and $\O_K = \Z[i]$. Then,  for some fixed values of the discriminant $\D$  of the binary quadratic Hermitian forms we display a table that gives the reduced forms with that given discriminant.

In Part 3, we provide a reduction  algorithm for binary forms of degree $n$ which is based on the basic theory of quadratic forms developed in Part 1 and 2.  For any binary form $f(X, Z)$ we   define the Julia invariant and Julia quadratic (covariant). We show that the Julia quadratic is a positive definite binary (Hermitian) quadratic form and then  the image of the  Julia quadratic under the zero map is a point in the upper half plane $\H_2$ ($\H_3$).  The degree $n$ binary form is called reduced if the image of the zero map is in the fundamental domain $\F$.

%


\bigskip

\noindent \textbf{Notation} Throughout this paper $k$ denotes a field not necessarily algebraically closed, unless otherwise stated. $K$ is   an algebraic number field and $\O_K$ its ring of integers. The discriminant of $K$ is denoted by $d_K$ while the discriminant of a polynomial $f$ of a binary form $F(X, Z)$ is denoted by $\D$. The Riemann sphere or the projective line are denoted by $\P^1$ and when the field needs to be pointed out, we will use $\P^1 (k)$ instead.

\bigskip

\newpage
\noindent \textbf{Part 1:  Binary quadratic forms} 

\vspace{3mm}

In this lecture, we give a brief description of the classical theory of binary quadratic forms with real coefficients. As it will be seen in Part 2 and Part 3 of these lectures, the quadratic forms with real coefficients will play a crucial role in the general reduction theory of binary forms.
\section{Quadratic forms over the reals} \label{binary-quad}
In this section we present some basics about binary quadratic forms.  For more details see \cite{lam}. Some of the results are elementary results from linear algebra and the proofs can be found in any linear algebra textbook; see     \cite{shilov} for a classical point of view with some emphasis on binary forms. 

\begin{defn} A \textbf{quadratic form over $\R$} is a function $Q: \R^n \to \R$ that has the form $Q(\textbf{x})= \textbf{x}^TA\textbf{x}$  where $A$ is a symmetric $n \times n$ matrix called the \textbf{matrix of the quadratic form}. 
\end{defn}

Two quadratic form $F(X, Z)$ and $G(X, Z)$ are said to be \textbf{equivalent over $\R$} if one can be obtained from the other by linear substitutions. In other words, 
\[ G(X, Z) = F(aX+bZ, cX+dZ),\]
for some $a, b, c, d \in \R$. In Section~\ref{bin-forms} we will define the equivalence for any degree $d$ binary forms over any field $k$.

\begin{lem} Let $F$,  $G$ be quadratic forms and $A_F$,  $A_G$ their corresponding matrices. Then, $F \sim G$ if and only if $A_F$ is similar to   $A_G$. 
\end{lem}


From now on the terms quadratic form and a symmetric matrix will be used interchangeably.

\begin{defn}
Let $Q(\textbf{x})=\textbf{x}^TA\textbf{x}$ be a quadratic form. 

i) The binary quadratic form $Q$ is \textbf{positive definite} if $Q(\textbf{x})>0$ for all nonzero vectors $\textbf{x} \in \R^n$, and $Q$ is \textbf{positive  semidefinite} if $Q(\textbf{x})\geq 0$ for all  $\textbf{x} \in \R^n$.  

ii) The binary quadratic form $Q$ is said to be \textbf{negative definite} if $Q(\textbf{x})<0$ for all nonzero vectors $\textbf{x} \in \R^n$, and $Q$ is \textbf{negative semidefinite} if $Q(\textbf{x})\leq 0$ for all  $\textbf{x} \in \R^n$.  

iii) $Q$ is \textbf{indefinite} if $Q(\textbf{x})$ is positive for some  $\textbf{x}$'s in  $\R^n$, and negative for others. 
\end{defn}

\begin{thm}\label{eigen}
Let $A$ be an $n \times n$ symmetric matrix, and suppose that $Q(\textbf{x})=\textbf{x}^TA\textbf{x}$. Then, 

i) $Q$ is  positive definite exactly when $A$ has only positive eigenvalues.

ii) $Q$ is negative definite exactly when $A$ has only negative eigenvalues.

iii) $Q$ is indefinite when $A$ has  positive and negative eigenvalues.
\end{thm}

\begin{proof}
Since $A$ is symmetric, there exist matrices $P$ and $D$ such that $P^TAP=D$, where the columns of $P$ are orthonormal eigenvectors of $A$ and the diagonal entries of $D$ are the eigenvalues $\lambda_1, \dots, \lambda_n$
of $A$. Since $P$ is invertible, for a given $\textbf{x}$  we can define $\textbf{y} =P^{-1} \textbf{x}$, so that $\textbf{x}=P\textbf{y}$. Then
\[\begin{split}
Q(\bx) &= \bx^TA \bx=(P\textbf{y})^TA(P\textbf{y}) = \by^T(P^TAP)\by = \by^TD\by=\lambda_1y_1^2+ \cdots + \lambda_ny_n^2
\end{split}
\]
If the eigenvalues are all positive, then $Q(\bx) >0$ except when $\by = \textbf 0$, which implies $\bx = \textbf 0$. Hence $Q$ is positive definite. On the other hand, suppose that $A$ has a nonpositive eigenvalue say, $\lambda_1 \leq 0$. If $\by$ has $y_1=1$ and the other components are $0$, then for the corresponding $\bx \neq \textbf{0}$ we have 
\[Q(\bx) =\lambda_1 \neq 0\]
so that $Q$ is not positive definite. The proof of ii) and iii) follow in the same way.
\end{proof}

The above definitions of positive definite carry over to matrices and they are found everywhere in the linear algebra literature.

\begin{defn}
A symmetric $n \times n$ matrix $A$ is \textbf{positive definite} if the corresponding quadratic form $Q(\textbf{x})=\textbf{x}^TA\textbf{x}$  is positive definite. Analogous definitions apply for \textbf{negative definite} and \textbf{indefinite}. 
\end{defn}

\begin{thm}[Spectral Theorem] 
A matrix $A$ is orthogonally diagnosable, i.e. there exists an orthogonal matrix $P$ and a diagonal matrix $D$ such that  $A=PDP^{-1}$,  if and only if $A$ is symmetric. 
\end{thm}


The following remarks are immediate consequence of the above. 

\begin{rem}
i) All eigenvalues of a symmetric matrix $A$ are real.

ii) Each eigenspace of a symmetric matrix $A$ has dimension equal to the multiplicity of the associated eigenvalue.
\end{rem}

\begin{thm}If $A$ is a symmetric positive definite matrix, then $A$ is nonsingular and $\det(A) >0$.
\end{thm}

\begin{proof}It is easy to check that $A$ is nonsingular.  Now let us show $\det(A)>0$. 
Suppose that $A$ has eigenvalues $\lambda_1, \dots, \lambda_n$ that are all real numbers. Then,  
\[\det(A) =\lambda_1 \lambda_2 \cdots \lambda_n.\] 
But from Theorem~\ref{eigen} all eigenvalues are positive, so is their product. Hence, $\det(A) >0$. 
\end{proof}

The proof of the following theorems is elementary and we skip the details.

\begin{thm}
A symmetric positive definite matrix $A$ has leading principal sub-matrices $A_1, A_2, \dots, A_n$ that are also positive definite. 
\end{thm}

Indeed, one can prove the following.

\begin{thm}
A symmetric matrix $A$ is positive definite  if and only if the  leading principal sub-matrices satisfy 
\begin{equation}\label{eq1}
\det(A_1) >0, \,  \det(A_2)>0, \dots, \, \det(A_n) >0.
\end{equation}
\end{thm}

The following is a well known result of basic linear algebra. We skip the proofs since they can be found in any textbook of linear algebra. 

\begin{thm} 
A symmetric matrix $A$ that satisfies Eq. ~\eqref{eq1} can be uniquely factored as $A= LDL^T$, where $L$ is a lower triangular matrix with 1's on the diagonal, and $D$ is a diagonal matrix with all positive diagonal entries. 
\end{thm}

Extremal properties of quadratic forms, which are of particular interest on this paper, and other interesting topics can be found on \cite[pg. 276-279]{shilov}. 

Next we will develop some of the main concepts needed to discuss the reduction of quadratic forms which will lead us to the general theory of the reduction of binary forms of any degree.

\section{The modular group and the upper half plane} \label{modular-group}

In this section   we will describe the action of $GL_2(\C)$ on the Riemann sphere, and we will show that this action has only one orbit. 

 Let $\P^1$ be the Riemann sphere and $ GL_2(\C)$  the group of $2 \times 2$ matrices with entries in $\C$. The group $GL_2(\C)$ acts on $\P^1$  by linear fractional transformations  as follows   
\begin{equation}\label{action-over-C}
\left( \begin{matrix} \a & \, \b \\ \g&\,  \d  \end{matrix} \right) z= \frac{\a z+\b}{\g z+\d}
\end{equation}
where $\left( \begin{matrix} \a & \b\\ \g& \d  \end{matrix} \right) \in  GL_2(\C)$ and $z \in \P^1$.  It is easy to check that  this is a group action. If a group $G$ acts on a set $S$, we say that $G$ \textbf{acts transitively} if for each $x,y \in S$ there exists some $g \in G$ such that $g(x)=y$.

\begin{lem}
$  GL_2(\C)$ action on $\P^1$ is a transitive action, i.e has only one orbit. Moreover, the action of $ SL_2(\C)$ on $\P^1$ is also transitive. 
\end{lem}

\begin{proof}

For every $z \in \C$,  
\[ \left( \begin{matrix} z & \, \, z-1\\1& \, \,  1  \end{matrix} \right) \infty= z\]
and  $\left| \begin{matrix} z & \, z-1\\1& 1  \end{matrix} \right|=1$.  So the orbit of infinity passes through all points.
\end{proof}

For the rest of this section we will consider the action of $SL_2(\R)$ on the Riemann sphere. Notice that this action   is not transitive, because as we will see below for $M= \left( \begin{matrix} \a & \,  \b\\ \g&\,  \d  \end{matrix} \right) \in GL_2(\R)$ we have
\[ \im \left(M z\right) = \frac{(\a \d -\b \g) \im  z}{|\g z+\d|^2} . \]
Hence, $z$ and $Mz$ have the same sign of imaginary part when $\det(M)=1$. Therefore we restrict this action only to one half-plane.   Let  $\H_2$ be the complex upper half plane, i.e 
\[\H_2 =\left \{ z = x+ i y \in \C \, \left | \frac{}{}\right. \, y>0\right \} \subset \C. \]
 The group $SL_2(\R)$  acts on $\H_2$ via linear fractional transformations. In the following lemma we prove that this action is transitive.
 
\begin{lem}
i)  The group $SL_2(\R)$ preserves $\H_2$ and acts transitively on it, further for $g \in SL_2(\R)$ and $z \in \H_2$ we have
\[\im  (gz) = \frac{\im  z}{|\g z +\d|^2}  \]
ii) The action of  $SL_2(\R)$  on $\P^1$ has three orbits,  namely $\R \cup \infty$, the upper half plane, and the lower-half plane.

\end{lem}

\begin{proof} Let us first prove that $\H_2$ is preserved under an $SL_2(\R)$ action.  Consider
\[ \left( \begin{matrix} \a & \b\\ \g& \d  \end{matrix} \right) \cdot z = \frac{\a z+ \b}{\g z+\d}\]
We want to find $\im  \left(  \frac{\a z+\b}{\g z+\d}  \right)$.  But  $\g z + \d= \g x+iy+\d = \g x+\d+ \d iy$, therefore it's conjugate is $(\d x+\d) -i\d y= \d \bar z +\d$ and 
\[(\d z+\d)(\d \bar z+ \d)= |\d z+\d|^2=(\d x+\d)^2+(\d y)^2.\] 
Hence,
\[
 \begin{split}
\frac{\a z+\b}{\d z+\d}&=  \frac{\a z+\b}{\d z+\d} \cdot  \frac{\d \bar z + \d}{\d \bar z+\d} =  \frac{(\a z+\b)(\d\bar z +\d)}{|\d z+\d|^2 }= \frac{\a \d z \bar z + \b \d + \a \d z + \b \d \bar z }{|\d z+\d|^2}\\
& =\frac{\a \d |z|^2 + \b \d + \a \d x + \a \d iy  + \b \d x - \b \d iy }{|\d z+\d|^2}\\
&= \frac{\a \d |z|^2 + \b \d + \a \d x +  \b \d x }{|\d z+\d|^2} + \frac{i(\a \d  - \b \d)y}{|\d z+\d|^2}
\end{split}
\]
Therefore we see that 
\[ \im \left( g z \right) = \frac{(\a \d - \b \d) \im  z}{|\d z+\d|^2} = \frac{\im  z}{|\d z+\d|^2} > 0 .\]
To show that $SL_2(\R)$ action on $\H_2$ is transitive, pick any $a+ib \in \H_2$. Then  if $g \in SL_2(\R)$ such that
\[g= \left( \begin{matrix} a & \, b\\0& \, 1 \end{matrix} \right): z \to a+ b z \]
we have  $g(i)=a+ib$. Thus the orbit of $i$ passes through all points in  $\H_2$ and so $SL_2(\R)$ is transitive in $\H_2$. 

ii) The result  is obvious from  above. 

\end{proof}


 Recall that a group action $G \times X \to X$ is called  \textbf{faithful} if there are no group elements $g$, except   the identity element, such that $g x=x$ for all $x \in X$.  The group  $SL_2(\R)$ does not act faithfully on $\H_2$ since  the elements $\pm I \in SL_2(\R)$
act trivially on $\mathcal H_2$. Hence, we consider the above action as  $PSL_2(\R)= SL_2(\R) /\{\pm I\}$ action.  This group acts faithfully on $\H_2$.  


\subsection{ The fundamental domain}\label{fund-domain}

Let $S$ be a set and $G$ a group acting on it. Two points $s_1, s_2$ are said to be \textbf{$G$-equivalent}  if $s_2= gs_1$ for some $g \in G$.  For any group $G$ acting on  a set $S$ to itself we call a \textbf{fundamental domain} $\F$, if one exists,  a subset of $S$ such that any point in $S$ is $G$-equivalent to some point in $\F$, and no two points in the interior of $\F$ are $G$-equivalent.

The group $\Gamma =SL_2(\Z) /\{\pm I\}$ is called the \textbf{modular group}.  It is easy to prove that $\Gamma$ action on $\H_2$ via linear fractional transformations is a group action.  This action has a  fundamental domain  $\F$ 
\[\F= \left\{\frac{}{} z \in \mathcal H_2  \left| \frac{}{} \right.    \, |z| \geq 1  \, \text{ and } \, |Re(z)| \leq 1/2 \right\}\]
as proven in the following theorem.

\begin{thm}\label{fund}

i) Every $z \in \H_2$ is $\Gamma$-equivalent to a point in $\F$. 

ii) No two points in the interior of $\F$ are equivalent under $\Gamma$. If two distinct points $z_1, z_2$ of $\F$ are equivalent under $\Gamma$ then $Re(z_1) = \pm 1/2$ and $ z_1= z_2 \pm 1$ or $|z_1|=1$ and $z_2 =-1/z_1$.

 iii) Let $z \in \F$ and $I(z) = \{ g \, | \, g \in \Gamma, \, gz =z\}$ the stabilizer of $z \in \Gamma$. One has $I(z) = \{1\}$ except in the following cases: 

  $z=i$, in which case $I(z)$ is the group of order 2 generated by $S$;
  
  $z=\rho=e^{2\pi i/3}$, in which case $I(z)$ is the group of order 3 generated by $ST$;
  
  $z=- \overline \rho=e^{\pi i/3}$, in which case $I(z)$ is the group of order 3 generated by $TS$.
\end{thm}


\begin{proof}

i)  We want to show that for every $z \in \mathcal H_2$, there exists $g \in \Gamma$ such that $gz \in \F$.  Let $\Gamma^\prime$ be a subgroup of $\Gamma$ generated by  
\[S= \left(\begin{matrix} 0  &-1 \\ 1 &0 \end{matrix} \right): z \to - \frac 1 z \quad \text{and} \quad  T=\left(\begin{matrix} 1 \quad 1 \\ 0 \quad 1 \end{matrix} \right): z \to z+1.\] 
Note that when we apply  an appropriate $T^j$ to $z$  then we can get a point equivalent with $z$ inside the stripe $-\frac 1 2 \leq Re(z) \leq \frac 1 2$.  If the point lands outside the unit circle then we are done, otherwise we can apply $S$ to get it outside the unit circle and then apply again an appropriate $T^n$ to get it inside the stripe  $-\frac 1 2 \leq Re(z) \leq \frac 1 2$. 

Let  $g  \in \Gamma^\prime$. We have seen  that $\im  (gz) = \frac{\im  z}{|cz +d|^2}.$
Since, $c$ and $d$ are integers, the number of pairs $(c, d)$ such that $|cz + d|$ is less then a given number is finite. Hence, there is some  $g=\left(\begin{matrix} a  &b \\ c &d \end{matrix} \right) \in \Gamma^\prime$ such that $ \im  (gz)$ is maximal($|cz + d|$ is minimal). 

Without loss of generality, replacing $g$ by $T^ng$ for some $n$ we can assume that $gz$ is inside the strip $-\frac 1 2 \leq Re(z) \leq \frac 1 2$.  If $|gz| \geq 1$ we are done, otherwise we can apply $S$. Then, 
\[ \im  (Sgz) = \frac{\im (gz)}{|gz +0|^2}= \frac{\im ( gz)}{|gz|^2} >\im (gz) .\]
But this contradicts our choice of $g \in \Gamma^\prime$ so that $ \im  (gz)$ is maximal. 


ii, and iii)  Suppose $z_1, z_2 \in \F$ are $\Gamma$-equivalent. Without loss of generality assume $ \im  (z_1) \geq \im  (z_2)$. Let $g=\left(\begin{matrix} a  &b \\ c &d \end{matrix} \right) \in \Gamma$  be such that $z_2 = gz_1$.  Since
\[ \im  (gz_1) = \frac{\im (z_1)}{|cz +d|^2}.\]
we get $|cz + d|\leq 1$.  But $z_1 \in \F$, $d \in \Z$, and  $ \im  (z_1) \geq  \frac{ \sqrt 3} 2$ hence the inequality does not hold for $|c| \geq 2$, i. e  $c=0, \pm 1$.

\textbf{Case 1:} $c=0$.   Since $ad-bc=0$ and $c=0$, we have $a, d =\pm1$ and $g= \pm \left(\begin{matrix} 1  &b \\ 0 &1\end{matrix} \right)$.  Since $Re(z_1)$ and $Re(z_2)$ are both between $-\frac 1 2$ and $\frac 1 2$, this implies either $b=0$ and  $g= \pm \left(\begin{matrix} 1  &0 \\ 0 &1\end{matrix} \right)$ or $b=\pm1$ and  $g=  \left(\begin{matrix} 1  &\pm 1 \\ 0 &1\end{matrix} \right)$  in which case either $Re(z_1)=\frac 1 2$ and $Re(z_2) = - \frac 1 2$, or the other way around. 

\textbf{Case 2:} $c=1$. Since $|1z_1+d|<1$, then $d=0$ except when $z_1=\rho$, or $-\overline \rho$ in which cases $d=0,1$ and $d=0, -1$. 

Let us first consider the case $c=1$, $d=0$. In this case $z_1$ is in the unit circle since otherwise  $|1z +0| \leq 1$ is not fulfilled, and since   $ad-bc=1$, we have $b=-1$ and $g= \pm \left(\begin{matrix} a  &\, -1 \\ 1 &0 \end{matrix} \right): z_1 \to a -\frac 1 {z_1} $.   The case $|a| > 1$  is not possible, since $z_1$ and $gz_1$ are both in $\F$. 

If, $a=0$, $z_1, z_2$  are symmetrically located on the unit circle with respect to the imaginary axis.  And for $a =\pm 1$, $g= \pm \left(\begin{matrix} \pm1  &\, -1 \\ 1 &0 \end{matrix} \right)=\pm T^{\pm 1}S$ from case 1 we have that  $Re(z_1)=\frac 1 2$ and $Re(z_2) = - \frac 1 2$, or the other way around i. e. $z_1, z_2=\rho, - \overline \rho$. 

The case $z=\rho$, $d=1$ gives $a-b=1$ and $g\rho= a -\frac{1}{1+\rho} = a+\rho$, hence $a=0, 1$; we can argue similarly when $z=- \bar \rho$, $d=-1$. 

Finally to prove the case when $c=-1$,  we just need to change the signs of $a, b, c, d$.






\end{proof}

The following corollary is obvious. 

\begin{cor}\label{fund-cor}
The canonical map $\F \to \mathcal H_2/\Gamma$ is surjective and its restriction to the interior of $\F$ is injective. 
\end{cor}

The following theorem determines the generators of the modular group and their relations. 

\begin{thm}
The modular group  $\Gamma$ is generated by  $S= \left(\begin{matrix} 0  &-1 \\ 1 &0 \end{matrix} \right)$  and     $ T=\left(\begin{matrix} 1 \quad 1 \\ 0 \quad 1 \end{matrix} \right)$, where $S^2=1$ and $(ST)^3=1$. 
\end{thm}

\begin{proof}  Let $\Gamma^\prime$ be a subgroup of $\Gamma$ generated by  
\[S= \left(\begin{matrix} 0  &-1 \\ 1 &0 \end{matrix} \right): z \to - \frac 1 z \quad \text{ and  } \quad  T=\left(\begin{matrix} 1 \quad 1 \\ 0 \quad 1 \end{matrix} \right): z \to z+1.\] 
We want to show that $\Gamma$ is a subgroup of $\Gamma^\prime$. Assume $g \in \Gamma$. Choose a point $z_1$ in the interior of $\F$, and let $z_2 =gz_1\in \H_2$. From the definition of the fundamental domain we have that there exists a $g^\prime \in \Gamma^\prime$ such that $g^\prime z_2   \in \F$.  But $z_1$ and $g^\prime z_2$ of $\F$ are  $\Gamma$-equivalent, and one of them is in the interior of $\F$, hence from Theorem~\ref{fund} this points coincide and $g^\prime g=1$. Thus,  $g \in \Gamma^\prime$. 
\end{proof}

Note that $S^2=1$, so $S$ has order 2, while $T^k= \left(\begin{matrix} 1 \quad k \\ 0 \quad 1 \end{matrix} \right)$ for any $k \in \Z$, so $T$ has infinite order. Figure~\ref{fig1} represents some transformations of $\F$ by some elements of $\Gamma$. 

\begin{figure}[!h]
\centering
\includegraphics[width=7cm, height=4cm]{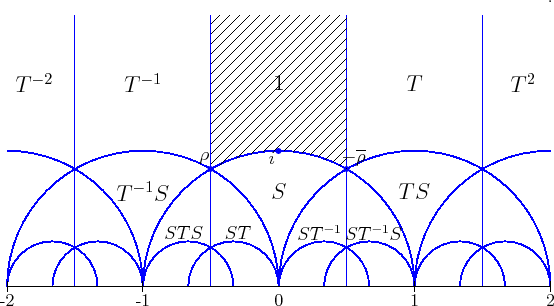}
\caption{The action of the modular group on the upper half plane.}
\label{fig1}
\end{figure}

For more details on the modular group and related arithmetic questions the reader can  see  \cite{serre} among others. 
\newpage
\section{The action of the modular group on the space of  positive definite binary quadratic forms}\label{modular-action-quadratic}

Let $Q(X, Z)= aX^2+ bXZ + cZ^2$ be a binary quadratic in $\R[X, Z]$. We will use the following  notation  to represent the binary quadratic, $Q(X, Z) = [a, b, c]$. The \textbf{discriminant} of $Q$ is $\D= b^2-4ac$ and  $Q(X, Z) $ is positive definite if $a>0$ and $\D <0$. Denote the set of positive definite binary quadratics with $BQF^+$, i.e.
\[BQF^+= \left \{ Q(X,Z) \in \R[X,Z] \, \left | \frac{}{}\right. \, Q(X,Z) \textit{ is positive definite }\right \}.\]
Let $SL_2(\R)$ act   as usual on the set of positive definite binary quadratic forms

 \[
\begin{split}
SL_2(\R)   \times BQF^+ &\to BQF^+\\
 \left(\begin{matrix} \a_1 \quad \a_2 \\ \a_3 \quad \a_4 \end{matrix} \right) \times  \left(\begin{matrix} X  \\ Z \end{matrix} \right)  &\to Q(\a_1X+\a_2Z,\a_3X+\a_4Z).
\end{split}
\]
We will denote this new form with $Q^M(X,Z) =a^\prime X^2+ b^\prime XZ+c^\prime Z^2$ where
\begin{equation}\label{Q^M}
\begin{split}
a^\prime &= a\a_1^2+ b\a_1\a_3+ c\a_3^2\\
b^\prime&=2(a\a_1\a_2 + c\a_3\a_4) + b(\a_1\a_4+ \a_2\a_3)\\\
c^\prime&=a\a_2^2+b\a_2\a_4+   c\a_4^2
\end{split}
\end{equation}
and 
\[\D^\prime=b^{\prime 2}-4a^\prime c^\prime=(\det M)^2 \D.   \]
Obviously, $\D$ is fixed under the $SL_2(\R)$ action and the leading coefficient of the new form $Q^M$ will be $Q^M(1, 0)= Q(a, c) >0$.   Hence,  $ BQF^+$ is preserved under this action. 

\subsection{The zero map}

Consider  the following map which is called the \textbf{zero map}
\begin{equation}\label{zero-map-real}
\begin{split}  
\z: BQF^+ &\to \H_2\\ 
[a, b, c] & \to \z(Q)  = \frac{-b + \sqrt \D}{2a} 
\end{split}
\end{equation}
where $Re(\z(Q)) =-\frac {b}{2a}$, and $\im (\z(Q)) =\frac{ \sqrt{|\D|}}{2a}$.   This map is a bijection since given $z=x+iy$, we can find $a, b, c$ such that $Q(X,Z)$ is positive definite  given as   $ [1, -2x, x^2+y^2]$. 
 
\begin{rem} Note that this map gives us   a one to one correspondence between positive definite quadratic forms and points in $\H_2$. 
\end{rem} 
 
\begin{defn} Let $G$ be a group and $X, Y$ two $G$-sets. A function $f: X \to Y$ is said to be \textbf{$G$-equivariant} if $f(gx)=g f(x)$, for all $g\in G$ and all $x\in X$.  This can be illustrated with the following diagram. 
\[
\begin{array}[c]{ccc}
X&\stackrel{g}{\longrightarrow}&X\\
\downarrow\scriptstyle{f}&&\downarrow\scriptstyle{f}\\
Y&\stackrel{g}{\longrightarrow}&Y
\end{array}
\]
Note that if one or both of the actions are right actions the $G$-equivariant condition must be suitably modified 
\[
\begin{split}
f(x\cdot g) &= f(x)\cdot g, \qquad \text{(right-right)}\\
f(x\cdot g) &= g^{-1} \cdot f(x), \quad \text{(right-left)}\\
f(g\cdot x) &= f(x)\cdot g^{-1},  \quad \text{(left-right)}
\end{split}
\]
\end{defn}
Let $\Gamma$ be the modular group acting on $\H_2$, and on $BQF^+$ as described above.  Then, the following theorem is true.

\begin{lem}\label{equivariant-real}
The root map $ \z:BQF^+ \to \H_2$ is a $\Gamma$-equivariant  map. In other words, $\z(Q^M)=M^{-1}\z(Q)$.
\end{lem}

\begin{proof}

Let $Q(X, Z)= aX^2+ bX + cZ$ with discriminant $\D$, and 
\[M= \left(\begin{matrix} \a_1 \quad \a_2 \\ \a_3 \quad \a_4 \end{matrix} \right) \in \Gamma\]
acting on it.  We want to show that $\z(Q^M)=M^{-1}\z(Q)$. We will prove the equivariance property only  for  the generators of $\Gamma$. From the root map,  equations~\eqref{Q^M}, and  using the fact that the discriminant is fixed  we have
\[
\begin{split}
\z(Q^M) &= \frac{-b^\prime +\sqrt{\D}}{2a^\prime} =\frac{ - (  2(a\a_1\a_2 + c\a_3\a_4) +  b(\a_1\a_4+ \a_2\a_3))+\sqrt{\D} }{2( a\a_1^2+ b\a_1\a_3+ c\a_3^2)}
\end{split}
\]
On the other side $M^{-1} \z(Q)$ is as follows
\[
\begin{split}
M^{-1}\z(Q)&= \left(\begin{matrix} \a_4 \quad - \a_2 \\- \a_3 \quad \a_1 \end{matrix} \right) \z(Q) =   \frac{\a_4 \z(Q) - \a_2}{  \a_1 -\a_3 \z(Q)} =\frac{\a_4  \frac{-b+ \sqrt \D}{2a}  - \a_2}{  \a_1 -\a_3  \frac{-b + \sqrt \D}{2a} } \\ 
&=\frac{\a_4  ( \sqrt \D -b)- 2a \a_2}{ 2a \a_1 -\a_3 ( \sqrt \D -b )} = \frac{  \a_4\sqrt \D -(2a\a_2 +b\a_4)}{(2a\a_1 +b\a_3) -\a_3 \sqrt \D}\\
&= \frac{[ \a_4\sqrt \D -(2a\a_2 +b\a_4)][(2a\a_1 +b\a_3) +\a_3 \sqrt \D]}{(2a\a_1 +b\a_3)^2 -\a_3^2 \D}\\
&= \frac{-4a^2\a_1\a_2-2ab\a_2\a_3-2ab\a_1\a_4 - b^2\a_3\a_4-2a\a_2\a_3\sqrt\D  }{(2a\a_1 +b\a_3)^2 -\a_3^2 \D}+ \\
&  \qquad + \frac{-b\a_3\a_4\sqrt\D + 2a\a_1\a_4\sqrt \D +b\a_3\a_4\sqrt\D +\a_3\a_4\D}{(2a\a_1 +b\a_3)^2 -\a_3^2 \D}
\end{split}
\]
If we let $M=T = \left(\begin{matrix} 1 & 1\\  0 & 1  \end{matrix}\right)$ we get $\z(Q^M)= M^{-1}\z(Q)= (-(2a+b)+\sqrt \D)/2a$,  and if we let $M$ equal the other generator of $\Gamma$, i.e  $M= S=\left(\begin{matrix} 0  &-1 \\ 1 &0 \end{matrix} \right)$, we get  $\z(Q^M)= M^{-1}\z(Q)= (b +\sqrt \D)/2c$. This completes the proof. 
\end{proof}

Note that the root $\z(Q)$ in the upper half-plane transforms via $M^{-1}$ into $(\a_4\z(Q)-\a_2)/ (\a_1-\a_3\z(Q))$, which is a also in the upper half-plane, because
\[\im (M^{-1} (\z(Q)))= \det(M^{-1}) \cdot  \frac{\im (\z(Q))}{  |\a_1-a_3\z(Q)|^2}.\]
%

\section{Reduction of positive definite quadratics}\label{reduction-quadratic} 

In this section we will define a reduced positive definite binary quadratic form, then we will give a reduction algorithm and at the end we will give an algorithm for counting reduced positive definite binary quadratic forms with a given discriminant. 

We denoted with $BQF^+$  the set of positive definite quadratics and we have defined an equivalence relation in this set.   Define $Q=[a, b, c]$ to be \textbf{reduced} if $\z(Q) \in \F$.  

The following theorem gives an arithmetic condition on the coefficients of a reduced positive definite binary quadratic. 

\begin{prop}\label{red-quad}
A positive definite quadratic form $Q \in BQF^+$ is  reduced if and only if $|b| \leq a \leq c$. 
\end{prop}

\begin{proof}  
Let $Q$ be a positive definite quadratic form with coefficients $[a,b,c]$. From the root map $ \z(Q)=\frac{-b +\sqrt \D}{2a}$.  By assumption  $\z(Q) \in \F$, i.e  
\[\frac{-1} 2 \leq Re(\z(Q)) \leq \frac 1 2 \quad \text{and} \quad |\z(Q)| \geq 1\]
Since, $\frac{-1} 2 \leq Re(\z(Q)) \leq \frac 1 2$ we have that   $\frac{-1} 2 \leq \frac{-b}{2a} \leq \frac 1 2$. Hence, $|b| \leq a$.  On the other side  since  $|\z(Q)| \geq 1$ we have
\[ 1 \leq  |\z(Q)| = \z(Q) \cdot \overline{\z(Q)}= \frac{(-b+\sqrt{\D})(-b-\sqrt{\D})}{2a \cdot 2a}=\frac{b^2-\D}{4a^2}= \frac{4ac}{4a^2}= \frac c a\]
Therefore, $|b| \leq a \leq c$
\end{proof}

The theorems that we will see for the rest of this section give us a reduction algorithm for positive definite binary quadratic forms, and they  will also be very  useful   for counting reduced forms with given discriminant $\D$. 

\begin{thm}\label{bound-b}
i) Let $Q$ be a reduced form with fixed discriminant $\D=-D$. Then, $b \leq \sqrt{D/3}$.

ii) The number of reduced forms of a fixed discriminant  $\D=-D$  is finite.
\end{thm}

\begin{proof} 

i) $Q$ is a positive definite binary quadratic with  fixed discriminant.  Since $Q$  is reduced from Proposition \ref{red-quad} we have that $|b| \leq a \leq c$.   Hence, 
\[4b^2 \leq 4ac = b^2 + D \]
i.e, $3b^2 \leq D$ and $b \leq \sqrt{D/3}$.

ii) From part i) there are only  finitely many possible $b$'s  and each of them determines a finite set of factorings $b^2 + D$ into $4ac$.  Hence, there are only finitely many candidates for reduced forms of fixed discriminant. 
\end{proof}

\begin{thm}
Every positive definite quadratic form $Q$ with fixed discriminant  is equivalent to a reduced form of the same discriminant. 
\end{thm}

\begin{proof}
Let $Q=[a, b, c]$ be a positive definite binary quadratic form with discriminant $\D$.  If this form is not reduced then  choose an integer $\d$  such that $|b +2c\d| \leq a$ ( choose $\d$ be the nearest integer to $- \frac{b}{2a}$)  and replace $[a, b, c]$ with $[a^\prime, b^\prime, c^\prime]=[a,b +2a\d,  a\d^2 + b \d +c]$.  The reduction transformation in this case is given by the  matrix 
\[ \left(\begin{matrix} 1 &\, \,  \d  \\  0  &\, \, 1  \end{matrix}\right)\]
which gives us $[a, b, c] \sim [a, b+2a\d, a\d^2 + b \d +c]$.

Then, if $c^\prime < a^\prime$ replace $[a, b, c]$ by  $[a^\prime, b^\prime, c^\prime]=[c, -b, a]$.  Since $a$,$c$  are positive integers  the process will terminate giving us the desired reduced form.

\end{proof}

With the exception of $[a, b, a] \sim [a, -b, a]$, and $[a, a, c] \sim [a, -a, c]$ no distinct reduced forms are equivalent. The proof is not difficult and can be found in \cite[pg. 15 ]{binary-quadratic}.  If we choose the reduced form to be the one that has a non-negative center coefficient then the following theorem hold. The proof is obvious from previous  theorems.  The interested reader can check \cite{binary-quadratic} for details. 

\begin{thm}\label{finite-discriminant}

i) Every form of discriminant $\D \leq 0$ is equivalent to a unique  reduced form. 
 
ii)  The number of reduced binary quadratic forms  for a given discriminant $\D$ is finite.
\end{thm}

We want to consider the  connection between the concept of a reduced form and the height of the $SL_2(\Z)$-equivalence class $[f]$ of a binary form $f$. Let us first recall the definition of the height as in \cite{nato-8}.

Let $f(X, Z) = aX^2+bXZ+cZ^2$ be a positive definite binary quadratic  form defined over $\R$.  From ~\cite{nato-8}, the height of   $f=[a, b, c]$   is  $H(f) = \max \{ |a|, |b|, |c| \}$.  If we consider $SL_2(\Z)$ acting on $BQF(\R)^+$  then in \cite{nato-8} we proved that there are only finitely many $f^\prime \in Orb(f)$ such that $H(f^\prime) \leq H(f)$ and we defined the height of the binary form to be 
\[ \tilde H(f) := \min \left\{ \frac{}{} H(f^\prime) | f^\prime \in  Orb(f), \, H(f^\prime) \leq H(f) \right\}.\]
%
Then, the following theorem holds.

\begin{thm} 
Let  $f(X, Z) = aX^2+bXZ+cZ^2$ be reduced (i.e., $|b| < a < c$).  Then,  $H( [f]) =c$.
\end{thm}

\begin{proof} 
We want to show that  given any  $M=\left(  \begin{matrix}   \a_1 & \, \a_2 \\ \a_3 & \, a_4 \end{matrix}  \right) \in SL_2( \Z)$  acting on $f(X,Z)$ we have that  $\max \{ |a_1|, |b_1|, |c_1| \} \geq c$, where $a_1, b_1, c_1$ are the coefficients of the new form $f^M$.  From~\eqref{Q^M} we have
\[
\begin{split}
a_1 &= a\a_1^2+ b\a_1\a_3+ c\a_3^2\\
b_1 &=2(a\a_1\a_2 + c\a_3\a_4) + b(\a_1\a_4+ \a_2\a_3)\\\
c_1 &=a\a_2^2+b\a_2\a_4+   c\a_4^2
\end{split}
\]

We will  prove it only for the generators of $SL_2( \Z)$, $S= \left(\begin{matrix} 0  &-1 \\ 1 &0 \end{matrix} \right)$  and     $ T=\left(\begin{matrix} 1 \quad 1 \\ 0 \quad 1 \end{matrix} \right)$.  First, let  $M=S$,  then we have $[a_1, b_1, c_1] = [c, -b, a]$  and if $M= T$ then  $[a_1, b_1, c_1] = [a, 2a+ b, a+b+c]$ and the result  is obvious. 

\end{proof}

\subsection{Counting binary quadratic forms with fixed discriminant}\label{quadratic-minimal-disc}

In Theorem~\ref{finite-discriminant}  we prove that  for a fixed   discriminant  $\Delta \leq 0$ there are finitely many reduced forms with discriminant $\Delta$. In this section we give an algorithm  to list such reduced forms with given discriminant.

\begin{alg}

\textbf{Input:} A binary quadratic form $F(X,Z) = aX^2 +bXZ+cZ^2$, where $a,b \in \Z$.  

\textbf{Output:} A binary quadratic form  $G$ equivalent to $F$, such that $G$ has    minimum height.\\

\textbf{Step 1:} Compute $\D_F=b^2-4ac$ for given $F(X, Z)$


\textbf{Step 2:}   Choose  $b$ such that  $b \leq \sqrt{\D/3}$.

\textbf{Step 3:}   For each picked $b$ find $a, c \in \Z$ such that 

\[ ac = \frac1 4(b^2- \D) \text{ and  } |b| \leq a \leq c.  \]

\textbf{Step 4:} Return the reduced forms  $[a, b, c]$.

\end{alg}

In Table~\ref{tab-1} we are listing (counting the number of) reduced forms with fixed discriminant $\D \equiv 1 \mod 4 $, $\Delta \leq 0$.   Note that $n$ represents the number of reduced forms with discriminant $\D$.   

From the equivalence classes of reduced quadratics there is one which has the smallest height. We call this class the special class and the corresponding height the minimal absolute height. For a generalization of this to degree $n$ binary forms see   \cite{nato-8}.

\begin{small}
\begin{table}[h]  \label{tab-1}
\centering
\begin{tabular}{|c| c| c|} 
\hline
$\Delta$  & Reduced form representative of classes  & n \\
\hline
-3 & [1, 1, 1] & 1\\
\hline
-7 & [1, 1, 2 ]  & 1 \\
\hline
-11 & [1, 1, 3] & 1 \\
\hline
-15 & [1, 1, 4],  \textbf{ [2, 1, 2] }& 2\\
\hline
-19 &  [1, 1, 5] & 1\\
\hline
-23 &  [1, 1, 6], \textbf{ [2,$\pm$ 1, 3]} & 3 \\
\hline
-27 & [1, 1, 7] & 1\\
\hline
-31 &  [1, 1, 8], \textbf{ [2, $\pm$ 1, 4] }& 3\\
\hline
-35 & [1, 1, 9], \textbf{[3, 1, 3]} & 2 \\
\hline
-39 &   [1, 1, 10], [2, $\pm$ 1, 5], \textbf{[3, 3, 4]}  & 4\\
\hline
-43 & [1, 1, 11] & 1\\
\hline
-47 &  [1, 1, 12],  [2, $\pm$1, 6], \textbf{[3, $\pm$1, 4] } &  5\\
\hline
 -51 & [1, 1, 13],  \textbf{[3, 3, 5] }& 2 \\
 \hline
 -55&  [1, 1, 14], [2, $\pm$ 1, 7], \textbf{ [4, 3,  4] }& 4\\
 \hline
 -59 &  [1, 1, 15], \textbf{ [3,$\pm$ 1, 5] }& 3 \\
 \hline
 -63 & [1, 1, 16],  [2, $\pm$ 1, 8],  \textbf{[4, 1, 4] } & 4\\
 \hline
 -67 & [1, 1, 17] & 1\\
 \hline
 -71 &     [1,1, 18], [2, $\pm$ 1, 9],  [3, $\pm$ 1, 6], \textbf{[4, $\pm$ 3, 5] }& 7\\
 \hline
 -75 & [1, 1, 19], \textbf{ [3, 3, 7] } & 2\\
 \hline
 -79 &   [1, 1, 20], [2, $\pm$ 1 , 10],  \textbf{[4,$\pm$ 1, 5]} &  5\\
 \hline
 -83 &   [1, 1, 21], \textbf{ [3, $\pm$1, 7] } & 3\\
 \hline
 -87 &  [1, 1, 22],  [2, $\pm$ 1, 11], [3, 3, 8],  \textbf{ [4,$\pm$ 3, 6] } & 6\\
  \hline
-91 & [1, 1,  23],  \textbf{[5,  3,  5]} & 2\\
\hline
-95 & [1, 1, 24],  [2, $\pm$ 1, 12],  [3, $\pm$1, 8], \textbf{[4, $\pm$ 1, 6]}, \textbf{ [5, 5, 6]}& 8\\
\hline 
-99 & [1, 1, 25], \textbf{ [5, 1, 5] }& 2 \\
 \hline
-103 & [1, 1, 26],  [2, $\pm$1, 13], \textbf{[4, $\pm$ 3, 7]}& 5\\
 \hline
-107 &  [1, 1, 27], \textbf{[3, $\pm$1, 9]} & 3\\
 \hline
-111 & [1, 1, 28], [2, $\pm$ 1, 14], [4, $\pm$1, 7], [3, 3, 10], \textbf{[5, $\pm$3, 6] }& 8\\
 \hline
 -115 & [1, 1,  29],  \textbf{[5, 5, 7] }& 2\\
  \hline
-119 &   [1, 1, 30],  [2, $\pm$1, 15],  [3, $\pm$1, 10],  \textbf{ [5, $\pm$1, 6]},  [4, $\pm$3, 8],  \textbf{[6, 5, 6]} & 10 \\
  \hline
  -123 & [1, 1, 31], \textbf{[3, 3, 11] }& 2\\
   \hline
   -127 &  [1, 1, 32], [2, $\pm$1, 16],  \textbf{[4, $\pm$1, 8]}&  5\\
    \hline
-131&     [1, 1, 33], [3, $\pm$1, 11],  \textbf{[5, $\pm$3, 7] }& 5\\
     \hline
  -135&  [1, 1, 34],  [2, $\pm$1, 17],  [4, $\pm$1, 9], \textbf{[5, 5, 8]} & 6\\
  \hline
  -139&  [1, 1, 35],  \textbf{ [5, $\pm$1, 7] }& 3\\
  \hline
  -143 &  [1, 1, 36],  [2, $\pm$1, 18], [3, $\pm$1, 12],  [4, $\pm$1, 9],  \textbf{[6, 1, 6]},  [6, $\pm$5, 7] & 10 \\
  \hline
 -147& [1, 1, 37],  \textbf{[3, 3, 13]} & 2\\
 \hline
 -151 &   [1, 1, 38],  [2, $\pm$1, 19],  [4, $\pm$1, 10],  \textbf{[5, $\pm$1, 8] }& 7\\
 \hline
 -155&   [1, 1, 39], [3, $\pm$1, 13], \textbf{[ 5, 5, 9] }& 4\\
 \hline
 -159& [1, 1, 40],  [2, $\pm$1, 20], [3, 3, 14],  [4, $\pm$1, 10],  [5, $\pm$1, 8],  \textbf{[6, $\pm$3, 7] }& 10 \\
 \hline
 -163& [1, 1, 41] & 1 \\
 \hline
\end{tabular} 
\vspace{.3cm}
\caption{Classes of quadratics with given discriminant}
\end{table}
\end{small}

\clearpage
\noindent \textbf{Part 2: Hermitian quadratic forms}

\vspace{3mm}

In these lecture, we give a brief description of binary quadratic Hermitian forms. We start with defining  binary Hermitian quadratic  forms defined over a  subring of $\C$.

\section{Reduction of Hermitian forms}\label{hermitian-forms}

In this section first we give some basics from linear algebra about    Hermitian matrices and Hermitian binary  forms. Then,  we describe $P SL_2(\C)$ action on the 3-dimensional hyperbolic space, denoted by $\H_3$ and  define the "zero" map which gives a one to one correspondence between positive definite Hermitian forms and points in $\H_3$.  At the end of the section we will define reduction of Hermitian forms and give an algorithm how to reduce them.

\begin{defn}
An $n\times n$ matrix $A$ with complex entries  is called Hermitian if $A^* =A$, where $A^*=\bar A^T$. 
\end{defn}

Recall that $\bar A$ is obtained from $A$ by applying complex conjugation to all elements and $A^T$ is the transpose of $A$.  By the definition we see that an Hermitian matrix is unchanged by taking it's conjugate transpose. Note that any Hermitian matrix must have real diagonal entries.

Let $R$ be a subring of $\C$ with $R=\bar R$, denote with $H(R)$ the set of $2\times 2$ Hermitian matrices, i.e
\[H(R)=\{ A \in M_2(R) \,  | \, A^*=A\}\]
A $2 \times 2$ matrix is in $H(R)$ if it is of the form 
\[ A = \left(\begin{matrix} a & b\\ \bar b & d  \end{matrix}\right)\]
where $a, d \in R \cap \R$ and $b \in R$.  Every matrix $A \in H(R)$ defines a \textbf{binary Hermitian form} with entries in $R$. If $A \in H(R)$ then the associated binary Hermitian form is the semi quadratic map
\[ Q: \C \times \C \to R\]
defined by 
\[Q(X,Z)= (X,Z)  \left(\begin{matrix} a & b\\ \bar b & d  \end{matrix}\right) (X, Z)^*= aX\bar X + bX\bar Z+ \bar b \bar X Z + d Z \bar Z.\]
The discriminant $\D(Q)$ of $Q \in H(R)$ is defined as $\D(Q)=\mbox{det }(Q)=ad-|b|^2$. A binary Hermitian form $Q \in H(R)$ is \textbf{positive definite} if $Q(X,Z) >0$ for every $(X,Z) \in \C \times \C \setminus \{0, 0\}$. $Q$ is called \textbf{negative definite} if $-Q$ is positive definite and \textbf{indefinite} if $\D(Q) <0$.  Denote with $H(R)^+$ the set of positive definite Hermitian forms, i.e
\[H(R)^+=\{Q \in H(R)\, |\,  Q \,  \textit{is positive definite}\}\]
 If $a\neq 0$, then 
\[Q(X,Z)= a \left(\left| X+ \frac{bZ}{a}   \right|^2+\frac{\D}{ a^2} |Z|^2 \right).\]
Hence,  $Q\in H^+(R)$ if and only if $a>0$ and $\D > 0$.

\subsection{Upper half space and the binary Hermitian forms}

Now we  describe the 3-dimensional hyperbolic space  $ \mathcal H_3$ and the action of $P SL_2(\C)$ on $ \mathcal H_3$.  Let 
\begin{equation}\label{hyperbolic-per-half-space}
\begin{split}
 \mathcal H_3:&= \C \times (0, \infty) =\{(z,t) | z \in \C, t >0\} =\left \{(x, y, t)| x, y \in \R, t>0\right\}
 \end{split}
 \end{equation}
A point $P \in \H_3$ is given as, $P=(z,t)=(x, y, t)=z+tj$ where $z=x+iy$ and $j=(0, 0, 1)$.  The group $P SL_2(\C)$ has a natural action on $\H_3$. Let $M =\left(\begin{matrix} \a & \b\\ \g & \d  \end{matrix}\right)$, and $P=z+tj$ a point in $\H_3$. Then,  $P SL_2(\C) $ acts on $\H_3$   via linear fractional transformation as follows
\[ M \times P  \to \frac{\a P+\b}{\g P+\d}\]
More explicitly  we have $M(z+tj)=z^*+t^*j \in \H_3$ where
\[
\begin{split}
z^*&=\frac{(\a z+\b)(\bar \g \bar z + \bar \d)+\a\bar \g t^2}{|\g z+\d|^2+|\g|^2t^2}\\
t^*&=\frac{t}{|\g z+\d|^2+|\g|^2t^2}
\end{split}\]

The action of $P SL_2(\C)$  on $\H_3$ leads to an action of $ SL_2(\C)$. 
\subsection{$ GL_2(\C)$ action on the set of  Hermitian forms}

The group $GL_2(R)$, where $R \subset \C$, as in Section~\ref{hermitian-forms}, acts on $H(R)$  as follows
\begin{equation}\label{action}
\begin{split}
GL_2(R) \times H(R) &\to H(R)\\
(M, Q) & \to M^\star Q M
\end{split}
\end{equation}
for $M \in GL_2(R)$ and $Q\in H(R)$. We can define in an analogue  way an $SL_2(R)$-action on $H(R)$.  Note that if  $A$ is the Hermitian matrix of $Q$ then the Hermitian matrix of the new form is $M^\star A M$.  It is easy to show that 
\begin{equation}\label{disc-Q^M}
\D(M(Q)) =   | \det  \, M|^2 \cdot \D(Q).
\end{equation}

The group $GL_2(R)$ leaves $H^+(R)$  invariant since for $M= \left(\begin{matrix} \a &\quad  \beta\\  \g & \quad \d  \end{matrix}\right)$ and $Q \in H^+(R)$, from equation  \ref{disc-Q^M}  we have that $\D(M(Q)) > 0$  and also it is easy to check that the leading coefficient of $Q^M=Q(\a, \g)>0$.

The group $\R^{>0}$ acts on $H^+(\C)$ by scalar multiplication.  We will denote by $ \tilde H^+(\C)$ the quotient space $H^+(\C) / \R^{>0}$, and $[Q]$  the equivalence class of $Q$  in $\tilde H^+(\C)$.   The action given in \eqref{action}, of $GL_2(\C)$ on $H(\C)$, induces an action of $GL_2(\C)$ on $\tilde H^+(\C)$. 

The center of  $ SL_2(\C)$ acts trivially on $H(\C)$, so we get an induced action of $PSL_2(\C)$ on $H(\C)$  and $\tilde H^+(\C)$.

\begin{thm}\label{generators-SL_2(C)}
The group $ SL_2(\C)$ is generated by   $\left(\begin{matrix} 0  & \, \, -1 \\ 1 & \, \,  0 \end{matrix} \right)$ and $\left(\begin{matrix} 1 \quad a \\ 0 \quad 1 \end{matrix} \right)$  where $a \in \C$.  This generators act on $(z, t)$, a point in $\H_3$, as follows
\begin{equation}\label{group-act-1}
\left(\begin{matrix} 1 \quad \a \\ 0 \quad 1 \end{matrix} \right) :(z, t) \to (z+\a, t)
\end{equation}
and
\
\begin{equation}\label{group-act-2}
\left(\begin{matrix} 0  & \, \, -1 \\ 1 &\, \, 0 \end{matrix} \right):(z, t) \to \left(  \frac{-\bar z}{|z|^2+t^2}, \frac{t}{|z|^2+t^2} \right).
\end{equation}
\end{thm}

\begin{proof} Let $M =  \left(\begin{matrix} \a & \, \,  \beta\\  \g &\, \,  \d  \end{matrix}\right) \in SL_2(\C)$. Let $\g \neq 0$,  then  we can factor $M$ as follows
\[ \left(\begin{matrix} \a &\quad  \beta\\  \g & \quad \d  \end{matrix}\right) = \left(\begin{matrix} 1 &\quad  \a \g^{-1} \\  0 & 1  \end{matrix}\right)  \left(\begin{matrix} 0 & \, \,  -1 \\  1  & \quad 0 \end{matrix}\right) \left(\begin{matrix} \g & 0\\  0  & \quad  -\b +\a \g^{-1} \d \end{matrix}\right) \left(\begin{matrix} 1 &\quad  \g^{-1} \d \\  0 & 1 \end{matrix}\right) \]
Consider, $ \left(\begin{matrix} \a &\, \,  0\\  0  & \, \,  \d \end{matrix}\right)  \in SL_2(\C)$. Since, $\a \d = 1$ then there exist $x, y \in \C^\star$ such that $1 = \a \d = xy(xy)^{-1}$ then,
\[\left(\begin{matrix} \a & \quad  0\\  0  & \quad  \d \end{matrix}\right) =
\left(\begin{matrix} x &\, \,  0\\  0  & \quad  x^{-1}\end{matrix}\right)
\left(\begin{matrix} y & \, \,  0\\  0  & \quad y^{-1} \end{matrix}\right)
\left(\begin{matrix} (yx)^{-1} & \, \, 0\\  0  & \, \,  yx \end{matrix}\right)
\left(\begin{matrix} \d^{-1} & \quad 0\\  0  & \quad  \d \end{matrix}\right)
  \]
and 
\[  \left(\begin{matrix} \a & \quad 0\\  0  & \quad \a^{-1} \end{matrix}\right) = 
\left(\begin{matrix} 1 &  \quad -\a\\  0  & \quad 1 \end{matrix}\right)
\left(\begin{matrix} 0 & \quad -1\\  1  & \quad 0 \end{matrix}\right)
\left(\begin{matrix}  1 & \quad -\a^{-1}\\  0  &  \quad 1\end{matrix}\right)
\left(\begin{matrix} 0 &\quad  -1\\  1  & \quad 0 \end{matrix}\right)
\left(\begin{matrix} 1 & \quad -\a\\  0  & \quad 1 \end{matrix}\right)
\left(\begin{matrix} 0 &\quad  -1\\  1  & \quad 0 \end{matrix}\right)\]

If, $\g =0$ then we have 
\[ \left(\begin{matrix} \a & \quad \beta\\  \g & \quad \d  \end{matrix}\right) =
 \left(\begin{matrix} \a & \quad 0 \\  0 & \quad \d  \end{matrix}\right) 
 \left(\begin{matrix} 1 & \quad \a^{-1} \beta\\  0 & \quad 1  \end{matrix}\right)\]
 Hence, every matrix can be expressed in terms of $T$, and  $S$.

\end{proof}

Note that this theorem  holds if we replace $\C$ with any number field $K$. Now we define the "zero map" for Hermitian forms. 

\begin{defn}\label{zero-map-hermitian}
The map  $\z: H^+(\C) \to \mathcal{H}_3$ defined by 
\begin{equation}\label{zero-map-hermitian}
\z  \left(\begin{matrix} a & b\\ \bar b & d  \end{matrix}\right) \to- \frac b a + \frac{\sqrt{\D(Q)}} a  \cdot j
\end{equation}
is called the \textbf{"zero map"} for binary quadratic Hermitian forms.  Clearly $\z$ induces a map  $ \z:  \tilde H^+(\C) \to \mathcal{H}_3$. 
\end{defn}

Since $Q$ is positive definite we have that  $a>0$ and $\D >0$,  hence  $\z$ is well defined and continuous.   This map is a bijection since given $(z, t)\in \H_3$ we can find $Q_{z,t}=[1, -z, - \bar z, |z|^2+t^2]$, i.e.  
\[Q_{z,t}: (u, v) \to |u|^2 - zu \bar v-\bar z \bar u v   +(|z|^2+t^2)|v|^2   \]
Therefore, this map gives a one to one correspondence between equivalence classes of positive definite binary quadratic Hermitian forms and points in $\H_3$. The following theorem holds.

\begin{thm}\label{equivariant-hermitian}
The map   $\z:  \tilde H^+(\C) \to \mathcal{H}_3$ defined by 
\[ [ Q ] \to - \frac b a + \frac{\sqrt{\D(Q)}} a  \cdot j  \]
is a $ PSL_2(\C)$ equivariant,  i.e. $\z$ satisfies $\z (Q^M )= M^{-1} \z (Q)$ for every $M \in PSL_2(\C)$ and $Q \in \H^+(\C)$. 
 \end{thm}

\begin{proof}

We will prove the equivariance property only  for  the generators of $ PSL_2(\C)$.  Let $Q \in H^+(\C)$, and $ A = \left(\begin{matrix} a & b\\ \bar b & d  \end{matrix}\right)$ be the Hermitian matrix of $Q$, and denote with  $\D$ the discriminant of $Q$.   We want to show that $\z (Q^M) = M^{-1} \z(Q)$. 

Let $M = \left(\begin{matrix} 1 & \beta\\  0 & 1  \end{matrix}\right)$, where $\beta \in \C$. Denote with $N$ the Hermitian matrix of $ Q^M$, then 
\[ N =  M ^\star A M  =  \left(\begin{matrix} 1 & 0 \\  \b & 1  \end{matrix}\right) \cdot \left(\begin{matrix} a & b\\ \bar b & d  \end{matrix}\right) \cdot  \left(\begin{matrix} 1 & \b \\  0 & 1  \end{matrix}\right) = \left(\begin{matrix} a  & a \b + b\\   a \b +  \bar b   & \quad \b( a\b + b + \bar b) +d  \end{matrix}\right)  \]
and 
\[\z (N )= \left(-  \frac{a\b + b}{a} , \frac{\sqrt \D}{a}  \right)\]
Now let us compute $M^{-1} \z(Q)$ and compare the two.  We know that $ \z(Q) = \left( - \frac{b }{a} , \frac{\sqrt \D}{a}  \right)\in \H_3$  and from equation~\eqref{group-act-1} we have 
\[M^{-1} \z(Q) =   \left(\begin{matrix} 1 & -\b \\  0 & 1  \end{matrix}\right) \cdot \left( - \frac{b }{a} , \frac{\sqrt \D}{a}  \right) =  \left(  -\frac{b }{a}  - \beta , \frac{\sqrt \D}{a}  \right)\]
We prove it the same way for $M=\left(\begin{matrix} 0  &\, \, -1 \\ 1 & \, \, 0 \end{matrix} \right)$.  The Hermitian matrix of the form $Q^M$ is 
\[ M^\star  A M=  \left(\begin{matrix} 0  &\, 1 \\ -1 & \, 0 \end{matrix} \right) \cdot \left(\begin{matrix} a & b\\ \bar b & d  \end{matrix}\right)   \cdot   \left(\begin{matrix} 0  & \, -1 \\ 1 & \, 0 \end{matrix} \right) =    \left(\begin{matrix} d  &\, - \bar b  \, \\ -b & \, a \end{matrix} \right)\]
and 
\[\z (M^\star A M )= \left(  \frac{\bar b}{d} , \frac{\sqrt \D}{d}  \right).\]
On the other side if we consider the action of $M=\left(\begin{matrix} 0  &-1 \\ 1 &0 \end{matrix} \right)$ on   $ \z (Q) = \left(  -\frac{b }{a} , \frac{\sqrt \D}{a}  \right)\in \H_3$ from equation~\eqref{group-act-2} we have 
\[M^{-1} \z (Q)= \left(\begin{matrix} 0  &1 \\ -1 &0 \end{matrix} \right) \cdot \left( - \frac{b }{a} , \frac{\sqrt \D}{a}  \right) =    \left(  \frac{- \left(-\frac{\bar b}{a}\right) }{\frac{|b|^2}{a^2} + \frac{ \D}{a^2}}, \quad \frac{ \frac{ \sqrt\D}{a}}{\frac{|b|^2}{a^2} + \frac{ \D}{a^2}}  \right)= \left( \frac{\bar b}{ d} , \frac{ \sqrt{\D}}{ d} \right) . \]
We get the desired result by simplifying the above and the equivariance of $\z$ follows. 
\end{proof}

\begin{rem} Note that Theorem~\ref{generators-SL_2(C)}, as well as  Theorem~\ref{equivariant-hermitian} is true if we replace $\C$ by any number field $K$ and the  proof in both cases   follows through in exactly  the same way.
\end{rem}


\section{The fundamental domains  over algebraic number fields}\label{any-number-field}

In this section we part from binary quadratic Hermitian forms briefly to describe some basic results about fundamental domains of number fields.   


The action  described in Equation~\eqref{action-over-C}    makes sense when $\C$ is replaced by any number field $K$, and gives a transitive group action of $GL_2(K)$ on $\P^1(K)=K \cup \infty$. We can prove, exactly in the same way as we did for the action of $ SL_2(\C)$ over $\P^1(\C)$, that the action of $SL_2(K)$ over $\P^1(K)$ is transitive. 

For analogues of $SL_2(\Z) \subset  SL_2(\R)$  and $BQF(\Z)^+ \subset BQF(\R)^+$  we need a discrete subring of $\C$.  Let $K$  be any number field, and  consider $SL_2(\O_K)$ where $\O_K$  is the ring of integers of $K$.  The generators of the special linear group $ SL_2(\O_K)$  with entries on $\O_K$ are  $\left(\begin{matrix} 0  &-1 \\ 1 &0 \end{matrix} \right)$ and $\left(\begin{matrix} 1 \quad a \\ 0 \quad 1 \end{matrix} \right)$  where $a \in \O_K$.

A \textbf{fractional ideal} is an $\O_K$-submodule $\mathfrak a$ contained in $K$ such that there exists an element $c \neq 0$  in $\O_K$  satisfying $c \mathfrak a \subset \O_K$.  Let $\mathfrak P$ be the subset of fractional ideals, then we write $\mathfrak a \sim \mathfrak b$ if there exists an element  $ \lambda \in K^* $ such that $\mathfrak a= (\lambda) \mathfrak b$, i.e. $\mathfrak a {\mathfrak b}^{-1}$ is a principal fractional ideal. The equivalence classes of fractional ideals form a finite group which we call the  \textbf{ideal class group}. It's order is usually denoted by $h_K$, and is called the \textbf{class number} of $K$.    Then, the following theorem holds.

\begin{thm} For a number field $K$, the number of orbits for $SL_2(\O_K)$ on $\P^1(K)$ is the class number of $K$. 
\end{thm}

 \begin{proof}

Let  $P=[x, y] \in \P^1(K)$, and we will denote a fractional ideal generated by $s$, $r$ as follows $\<s, r\>=s\O_K +r\O_K$. We want to  prove that  if $[x, y]$ and $[z, w]$ are in the same $SL_2(\O_K)$ orbit, then $\mathfrak a =\<x, y\>$ and $\mathfrak b =\<z, w\>$ (the fractional  ideals generated respectively from $x,y$ and $z,w$ are in the same ideal class.  From definition, we want to show that  exists an element  $ \lambda \in K^* $ such that $\mathfrak a= (\lambda) \mathfrak b$.

The fact that $[x, y]$ and $[z, w]$ are in the same $SL_2(\O_K)$ orbit means that there exists an $M = \left( \begin{matrix} \a_1 & \a_2\\ a_3& \a_4   \end{matrix}  \right)$  and $\lambda \in K^\star$ such that 
\[ \left( \begin{matrix} \a_1 & \a_2\\ a_3& \a_4   \end{matrix}  \right)  \left[ \begin{matrix} x \\ y \end{matrix}  \right]= \lambda  \left[ \begin{matrix} z \\ w \end{matrix}  \right].\]
Hence, 
\[ \begin{split}
\a_1x+\a_2y &=\lambda z\\
\a_3x+\a_4y&=\lambda w
\end{split}
\]
and we have  $\<\lambda z, \lambda w\> \subset \<x, y\>$.  Multiplying both sides of above with $M^{-1}$ we get the other inclusion 
\[ \begin{split}
x &= \a_4 \lambda z - \a_2 \lambda w \\
y &= \a_1 \lambda w -\a_3 \lambda z
\end{split}
\]
and we conclude that $[x, y]$ and $[z, w]$ are in the same $SL_2(\O_K)$ then they are equivalent as fractional ideals, $\<x, y\> = \lambda \<z, w\>$.

Let us prove the other direction.  Let $\<x, y\>$ and $\<z, w\>$  be in the same ideal class, then there exists an element $\lambda \in K^\star$ such that  $\<x, y\> = \lambda \<z, w\>$. We want to prove that the points $[x, y]$ and $[z, w]$ are in the same $SL_2(\O_K)$. Since points in $[z, w] \in \P^1(K)$,  i.e. $[\lambda z, \lambda w] = [z, w]$, without loss of generality we can assume $\lambda$to be one.  Under this assumption $\<x, y\>$ and $\<z, w\>$ are the same as fractional ideals. 

Let $\mathfrak a =(x, y)$, then $\mathfrak a^{-1}$ is a fractional ideal and hence has two generators assume $\mathfrak a^{-1} = (m, n)$. Then, $ \mathfrak a \mathfrak a^{-1} = 1 = \<x, y\>\<m,n\> =\<xm,  xn, ym, yn\>$. 

There exist $\a_1, \a_2, \a_3, \a_4 \in \O_K$ such that  
\[1=\a_1x m+\a_2 x n +\a_3 y m+ \a_4 y n = x(\a_1m+\a_2n) + y(\a_3m+\a_4n).\]
If we let $x^\prime = \a_1m+\a_2n \in \mathfrak a^{-1} $ and $y^\prime = \a_3m+\a_4n \in \mathfrak a^{-1}$ we can form a matrix $M =  \left( \begin{matrix} x & x^\prime \\ y & y^\prime   \end{matrix}  \right)$ with determinant 1 and entries in $\O_K$.

In the same way we can show that there exists a matrix $M^\prime = \left( \begin{matrix} z &\,  z^\prime \\ w &\,  w^\prime   \end{matrix}  \right)$ with determinant 1 and entries in $\O_K$. Consider the matrix $M {M^\prime}^{-1}$, 
\[M {M^\prime}^{-1}= \left( \begin{matrix} x & \, x^\prime \\ y& \, y^\prime \end{matrix}  \right)  \left( \begin{matrix} w^\prime &\,  -z^\prime \\   -w & \, z  \end{matrix}  \right) \]
which has determinant 1 and entries in $\O_K$, i.e is a matrix in $SL_2(\O_K)$ and 
\[
\begin{split}M {M^\prime}^{-1}[z, w]&= \left( \begin{matrix} x &\,  x^\prime \\ y& \, y^\prime \end{matrix}  \right)  \left( \begin{matrix} w^\prime &\,  -z^\prime \\   -w & \, z  \end{matrix}  \right)  \left( \begin{matrix} z \\ w \end{matrix}  \right) =  \left( \begin{matrix} x & \, x^\prime \\ y& \, y^\prime \end{matrix}  \right) \left( \begin{matrix} zw^\prime- z^\prime w\\ zw -wz\end{matrix}  \right) \\
&=  \left( \begin{matrix} x & \, x^\prime \\ y&\, y^\prime \end{matrix}  \right)  \left( \begin{matrix} 1 \\ 0 \end{matrix}  \right)\\
& = [x,y]
\end{split}
\]
Therefore, $[x, y]$ and  $[z, w]$ are $SL_2(\O_k)$ equivalent.

\end{proof}

An immediate corollary of the theorem is the following.

\begin{cor}
$SL_2(\O_K)$ acts transitively on $\P^1(K)$  if and only if $K$ has class number 1.
\end{cor}

Reduction theory for the case when $h_K=1$ and $h_K>1$ are significantly different.  We will consider only  the case when $h_K=1$. 

\section{Reduction theory of Hermitian forms}\label{reduction-hermitian}

Reduction of real binary forms with respect to the action of  $SL_2(\Z)$, as described in Section~\ref{reduction-quadratic},  may be extended to a reduction theory for binary forms with complex coefficients (Hermitian binary forms) under the action of certain discrete subgroups of $\C$.  In order to do that we need a discrete subring of $\C$ and then define the fundamental domain  of this action. 

In this section, we will consider the case when    $K=\Q( \sqrt \D) \subset \C$ is an imaginary quadratic number field of discriminant $\D <0$ a square-free integer, $d_K$ the discriminant of $K$, and $\O_K$ it's ring of integers which is a discrete subring of $\C$.


Let     $H(\O_K)$ denotes   the space of binary Hermitian forms with coefficients in $\O_K$,   $H^+(\O_K)$ denote  the set of positive  definite Hermitian forms with coefficients in $\O_K$, and   $H^-(\O_K)$ the set of indefinite Hermitian forms with coefficients in $\O_K$. It is easy to show that the \textbf{ "Bianchi group" }   $\Gamma = PSL_2( \O_K)$ acts on on $\H_3$, and also on $H^+(\O_K)$ preserving discriminants.  This action has a fundamental domain, which we will denote it with $\F_K$ and depends on $K$.  For small discriminant this was determined by Bianchi and others in the 19 century.  

Consider $PSL_2(\O_K)$ action on $ \H_3$, and define  the following 
\[\begin{split}
\mathcal B_K   &=   \left\{   \frac{}{}    z+rj   \in \H_3 \,  \left | \frac{}{} \right.  \, |cz+d|^2 +|d|^2 r^2 \geq 1,  \forall c, d  \in \O_K \text{ : } \< c, d \> = \O_K   \right\}\\
\mathcal P_K&=\left\{\frac{}{}z \in \C  \left | \frac{}{} \right. 0 \leq \re (z) \leq 1, \quad 0 \leq \im (z) \leq \sqrt{|d_K|}/2\right\}\\
F_K&= \mathcal P_K, \, \, \text{for } \, \,  \D\neq -3, -1\\
F_{\Q(i)}  &=  \left \{ \frac{}{} z \in \C  \left | \frac{}{} \right. \, 0 \leq |\re(z)| \leq \frac 1 2  , \, \, 0 \leq \im(z)  \leq \frac 1 2  \right\}\\
F_{\Q(\sqrt{-3})}& =   \left \{ \frac{}{}   z \in \C  \left | \frac{}{} \right. \, 0 \leq \re(z), \,  \frac{\sqrt3} 3 \re(z) \leq \im(z),   \, \im(z)  \leq \frac{\sqrt3} 3 \left(\frac{}{}1-\re(z)\right)  \right\}\\
& \,   \cup  \left \{ \frac{}{}  z \in \C  \left | \frac{}{} \right. \, 0 \leq \re(z) \leq \frac 1 2, \,  - \frac{\sqrt3} 3 \re(z) \leq \im(z)  \leq \frac{\sqrt3} 3  \re(z)  \right\}\\
\F_K&=\left\{\frac{}{}z+rj\in \mathcal B_K\,  \left | \frac{}{} \right.  \,  z \in F_K\right\}
\end{split}
\]

\begin{thm}
The set $\F_K$ is a fundamental domain for $PSL_2(\O_K)$. 
\end{thm}

\begin{proof}
For proof see ~\cite[pg 319]{egm}.
\end{proof}

The following definition is analog to the one of positive definite binary quadratic forms. 

\begin{defn}
A  positive definite Hermitian form $f \in H^+(\O_K)$  is  called a   \textbf{reduced   Hermitian form } if $\z(f) \in \F_K$. 
\end{defn}


\subsection{Counting binary quadratic Hermitian forms with fixed discriminant}

In this subsection, $K$ is an imaginary quadratic number field, as above, and $\O_K$ it's ring of integers. Let 
\[H(\O_K, \D )=\{ f \in H(\O_K) \, | \,  \D (f)=\D    \},\]
be the subspace of $H(\O_K)$ with fixed discriminant $\D$ and 
\[H^{\pm}(\O_K, \D )=\{ f \in H^{\pm}(\O_K) \, | \, \D (f)=\D \} \]
the subspace of $H^{\pm}(\O_K)$ of fixed discriminant.  Then, the following theorem holds.

\begin{thm}\label{thm2} 
Given  $\D \neq 0 \in \Z$, the number of reduced forms of  $H(\O_K, \D)$  is finite. 
\end{thm}

The proof can be found in ~\cite[pg. 411]{egm}.

\begin{cor}
For any $\D \in \Z$ with $\D \neq 0$ the set $H(\O_K, \D)$  (and $H^{\pm}(\O_K, \D)$) splits into finitely many $SL_2(\O_K)$ orbits. 
\end{cor}

\begin{proof}
This is an immediate consequence of Theorem~\ref{thm2}, and Theorem~\ref{equivariant-hermitian}  which says that every $f \in H(\O_K, \D)$, is $PSL_2(\O_K)$-equivalent to a reduced form. 

\end{proof}

For any $\D \in \Z$ with $\D \neq 0$ define 
\[ \tilde H(\O_K, \D)= SL_2(\O_K)\backslash H(\O_K, \D), \]
and denote by  $h(\O_K, \D) := \left| \frac{}{} \tilde H(\O_K, \D) \right|$,  where the number $h(\O_K, \D)$ is called the \textbf{class number of binary Hermitian forms of discriminant $\D$}.

We define the same way for positive definite Hermitian forms $\tilde H^+(\O_K, \D)= SL_2(\O)\backslash H^+(\O, \D)$ such that $h^+(\O_K, \D)= \left| \frac{}{} \tilde H^+(\O_K, \D) \right|$, and  $h^+(\O_K, \D)$ is called the \textbf{class number of positive definite binary Hermitian forms of discriminant $\D$}.  Note that for $\D >0$ we have that $h(\O_K, \D)=2h^+(\O, \D)$. 

Given $\O_K$ and the discriminant $\D$ it is always possible to compute the class number of positive definite binary Hermitian forms with given discriminant $\D$. Let us now consider the case when $K=\Q(i)$. Then, $d_K =-4$  and the  ring of integers is  the ring of Gaussian integers $\O_K=\Z[i]$ . 

\begin{lem}The fundamental domain  $\F_{\Q(i)}$ for $PSL_2(\Z[i])$  is   as follows 
\begin{equation}\label{fund-imaginary}
\F_{\Q(i)}=\left\{\frac{}{}z+rj\in \H_3 \,  \left | \frac{}{} \right. \,   0 \leq |\re(z)| \leq \frac 1 2   , \, \, 0 \leq \im(z)  \leq \frac 1 2 , \, \,   z \bar z + r^2 \geq 1  \right\}
\end{equation}
$\F_{\Q(i)}$ is a hyperbolic pyramid with one vertex at infinity and the other four vertices in the points $P_1= -\frac 1 2 + \frac{\sqrt 3} 2 \cdot j$,  $P_2= \frac 1 2 + \frac{\sqrt 3} 2 \cdot j$, $P_3= -\frac 1 2 (1+i) + \frac{\sqrt 2} 2 \cdot j$,  $P_4= -\frac 1 2 (i-1) + \frac{\sqrt 2} 2 \cdot j$.  Let 
\[ A = \left( \begin{matrix} 1 & \, 0 \\ 1 & \, 1  \end{matrix}\right), \quad B=\left( \begin{matrix}  0 & \, -1\\ 1 & \, 0  \end{matrix}\right), \quad C= \left( \begin{matrix} 1 & \, 0 \\ i & \, 1  \end{matrix}\right) \]
Then the following is a presentation for $PSL_2(\Z[i]$.
\[ \begin{split}
PSL_2(\Z[i] = \left \<  \frac{}{} \right.  A, B, C    \left | \frac{}{} \right. \,  (AB)^3&= B^2= ACA^{-1}C^{-1} = (BCBC^{-1} )^3 =\\ 
& = (BC^2BC^{-1})^2 = (ACBAC^{-1}B)^2=1 \left  \>  \frac{}{} \right.
\end{split}\]
\end{lem}

\begin{proof}
See \cite[pg. 325]{egm}
\end{proof}

We want to count the number of reduced positive definite  binary Hermitian forms with a fixed discriminant $\D$, i.e $h^+(\Z[ i ],  \D)$ . Let $f = \left(\begin{matrix} a & \, b\\ \bar b & \, c  \end{matrix}\right)$ be a positive definite binary quadratic  Hermitian form with coefficients in $Z[i]$ and  non-zero discriminant $\D$.   The binary quadratic Hermitian form $f$ is reduced if $\z(f) \in \F_{\Q(i)}$, i. e 
\[- \frac b a + \frac{\sqrt \D} a \cdot j \in \F_{\Q(i)}.\]  
If we let  $z= - \frac b a$ and $r=\frac{\sqrt \D} a$, from \eqref{fund-imaginary} we have $a \leq c$, $0 \leq  |\re(-b)| \leq \frac a 2$, and $0 \leq \im(-b) \leq \frac a 2$, and $a^2 \leq 2 \Delta$. 

By discreteness of $\Z[i]$, the elements $a$, and $b$ may take only finitely many values. The discriminant  $\D=ac-b \bar b $, hence $c$ is determined by $a$, and $b$. Therefore,   $c$ may take only finitely many values too.  

In the following table we listing (counting) the number of reduced  binary quadratic Hermitian forms with fixed discriminant.  To each tuple $[a, b, c]$ corresponds a binary quadratic Hermitian form 
\[Q(X,Z) =aX\bar X + bX\bar Z+ \bar b \bar X Z + c Z \bar Z.\]
In the first column is given the discriminant, in the second one the reduced forms $[a, b, c]$  with that given discriminant, and in the third column the number of reduced forms.

\begin{small}
\begin{table}[h]  \label{tab-1}
\centering
\begin{tabular}{|c| c| c|} 
\hline
$\Delta$  & Reduced form representative of classes  given  by $[a, b,c]$ & n \\
\hline
1 & [1, 0, 1]  & 1 \\
\hline
2 & [1, 0, 2],  [2, 0, 2], [2, $\pm$ 1-i, 2]  & 4\\
\hline
3 &  [1, 0, 3], [2, $\pm$ 1, 2], [2, -i, 2]&4 \\
\hline
4 &  [1, 0, 4], [2, 0, 2], [2, $\pm$1-i, 3] & 4 \\
\hline
5 & [1, 0, 5], [ 2, $\pm$1, 3], [2, -i, 3]   & 4 \\
\hline
6 & [1, 0,6], [2, 0, 3], [2, $\pm$1-i, 4]  & 4 \\
\hline
7 & [1, 0, 7], [2, $\pm$1, 4], [2, -i, 4], [3, $\pm$1-i, 3] & 6\\
\hline
8 &  [1, 0, 8], [2, 0, 4], [2, $\pm$1-i, 5],  [3, $\pm$ 1, 3], [3, -i, 3], [4,$ \pm$ 2 -2i, 4] & 9  \\
\hline
9 & [1, 0, 9], [2, $\pm$1, 5], [2, -i, 5], [3, 0, 3]   & 5 \\
\hline
10 &  [1, 0, 10], [2, 0, 5], [2, $\pm$1-i, 6], [3, $\pm$1-i, 4] & 6 \\
\hline
11 &  [1, 0, 11], [2, $\pm$1, 6], [2, -i, 6], [3, $\pm$ 1, 4], [3, -i, 4], [4, $\pm$ 2 -i, 4], [4, $\pm$1-2i, 4]&11 \\
\hline
12 & [1, 0, 12], [2, 0, 6], [2, $\pm$1-i, 7], [3, 0, 4]  & 5 \\
\hline
13 &  [1, 0, 13], [2, $\pm$1, 7], [2, -i, 7], [3, $\pm$1-i, 5], & 6 \\
\hline
14 &  [1, 0, 14], [2, 0, 7], [2, $\pm$1-i, 8], [3, $\pm$1, 5], [3, -i, 5], [4, $\pm$1-i, 4] & 9 \\
\hline
15 &[1, 0, 15], [2, $\pm$1, 8], [2, -i, 8],  [3, 0, 5], [4, $\pm$2, 4], [4,  -2i, 4], [4, $\pm$1-2i, 5],& 12 \\
& [4, $\pm2$-i, 5] &\\
\hline
16 & [1, 0, 16],[2, 0, 8], [2, $\pm$1-i, 9],  [2, $\pm$1-i, 6], [4, 0, 4], [4, $\pm$2, 5], [4, -2i, 5] & 12 \\
\hline
17 & [1, 0, 17], [2, $\pm$1, 9], [2, -i, 9], [3, $\pm$1, 6], [3, $\pm$i, 6], [5, $\pm$2-2i, 5]  & 10 \\
\hline
18 & [1, 0, 18], [2, 0, 9], [2, $\pm$1-i, 10], [3, 0, 6], [4, $\pm$1-i, 5], [4, $\pm$2, 5], [4, -2i, 5] &10\\
\hline
19&  [1, 0, 19], [2, $\pm$ 1, 10], [2, -i, 10], [3, $\pm$1-i, 7], [4, $\pm$1, 5], [4, -i, 5], [4, $\pm$1-2i, 6], & 13\\
&  [4, $\pm$2-i,6] &\\
\hline
\end{tabular} 
\vspace{.3cm}
\caption{Classes of binary quadratic Hermitian forms with given discriminant}
\end{table}
\end{small}

\newpage
\noindent \textbf{Part 3:  Reduction of binary forms of higher degree}

\vspace{3mm}

In Part 3 of these lectures we describe how Julia, and then Stoll-Cremona  developed reduction theory  for binary forms defined over $\R$, $\C$ of degree $n \geq 2$ using reduction theory  of binary quadratics, respectively Hermitians.


\section{Introduction to higher degree binary forms}\label{bin-forms}


Let  $k$ be   an algebraically closed field. In this section we define    binary forms of degree $n$  with coefficients in $k$ and the  action of $ GL_2(k)$ on the space of degree $n$ binary forms.  

Let $k[X,Z]$  be the  polynomial ring in  two variables and  let $B_n$ denote  the $(n+1)$-dimensional subspace  of  $k[X,Z]$  consisting  of homogeneous polynomials.
\begin{equation}
\label{eq-1} 
f(X,Z) = a_0 X^n + a_1X^{n-1}Z + \dots + a_nZ^n
\end{equation}
of  degree $n$. Elements  in $B_n$  are called  \textbf{binary  forms of degree $n$}.  Since  $k$ is an algebraically closed field, the binary form $f(X,Z)$   can be factored as
 \begin{equation}  f(X,Z)  = (z_1  X  -  x_1 Z) \cdots  (z_d  X  - x_d  Z)
=  \prod_{1 \leq  i \leq  d} det
\begin{pmatrix}
X&x_{i}\\Z&z_i\\
\end{pmatrix}
\end{equation}
The points  with homogeneous coordinates $(x_i, z_i)  \in \mathbb P^1 (k)$ are  called the \textbf{ roots  of the  binary  form} in Eq.~ \eqref{eq-1}.

The group $GL_2(k)$ acts by linear transformations   on the variables of  $ f \in k[X, Z]$. Let $M= \begin{pmatrix} \a &\b \\  \gamma & \d \end{pmatrix}  \in GL_2(k)$ and  $ f \in k[X, Z]$, then
\[
\begin{split}
GL_2(k) \times B_n(k) &\to B_n(k)\\
\left (M, f(X, Z)  \right) & \to f(\a X+\b Z, \gamma X+\d Z)
\end{split}
\]
This action of $GL_2(k)$  leaves $B_n$ invariant.   For $M  \in GL_2(k)$ it is easy to show that
\begin{center}
$M\left(f(X,Y)\right)  = \det(M)^ n  (z_1^\prime  X -  x_1^{\prime}  Z) \cdots (z_n^{\prime} X  - x_n^{\prime} Z)$.
\end{center}
where
\[
 \begin{pmatrix}   x_i^{\prime}  \\  z_i^{\prime}  \end{pmatrix}
= M^{-1}
\begin{pmatrix} x_i\\ z_i \end{pmatrix}
\]

$GL_2(k)$ action on $B_n$ induces a $SL_2(k)$ action on this set. It is well  known that $SL_2(k)$ leaves a bilinear  form (unique up to scalar multiples) on $B_n$ invariant.

\begin{defn} A non-zero degree $n \geq 3$ binary form is called \textbf{stable} if none of its roots has multiplicity $\geq \frac n 2$. 
\end{defn}

\section{Julia  invariant and covariant of a binary form} 

In 1917, Julia in his thesis \cite{julia-reduction} introduced an invariant of the action of $SL_2(\R)$ on binary forms. This invariant was used to define reduction theory for binary forms of higher degree.   In this section we define  Julia's invariant  for a binary form of degree $n \geq 2$. 

Let $f(X, Z) \in \R[X, Z]$ be a degree $n$ binary form  given as follows 
\[ f(X, Z) =a_0X^n+a_1X^{n-1}Z+ \cdots +a_nZ^n   \]
and suppose that $a_0 \neq 0$. Let the real roots of $f(X,Z)$ be $\a_i$,   for $1 \leq i \leq r$ and the pair of complex roots $\b_j$, $\bar \b_j$ for $1 \leq j \leq s$, where $r+2s =n$.  To obtain a representative point in the complex upper half plane,  construct a   quadratic form 
\[ Q_f (X,Z) =\sum_{i=1}^r t_i^2(X-\a_iZ)^2+ \sum_{j=1}^s 2u_j^2(X-\b_jZ)(X-\bar \b_jZ), \] 
where $t_i$, $u_j$ are real numbers that have to be determined.  The following lemma holds.

\begin{lem}\label{julian-quad-min}
i) $Q_f (X,Z) \in \R[X, Z]$ is a positive definite quadratic form

ii) There exists a unique tuple    $\mathfrak t_f= (t_1, \dots , t_n)$ which  make $\theta_0 (Q_f)$
\[ \theta_0 (Q_f) = \frac{a_0^2 (\D( Q_f))^{n/2}}{\prod_{i=1}^rt_i^2 \prod_{j=1}^s u_j^4  }\]
minimal.

%
%

\end{lem}

\proof
i) If we let $Z=1$  we have that 
\[\begin{split} Q_f (X,1) &=\sum_{j=1}^n t_j(X-\a_j)(X-\bar \a_j)= \sum_{j=1}^n t_j(X^2 - (\a + \bar \a) X + |\a|^2 )\\
&= \sum_{j=1}^n t_j(X^2 - 2 \re(\a) X + |\a|^2 )
\end{split}\] 
where $\a \in \C$. Computing the discriminant of $g=X^2 - 2 \re(\a) X + |\a|^2 $ we get $\Delta(g) = -4\im (\a)^2\leq 0$.  Since the $t_j$ are assumed to be positive and $\Delta < 0$, then $g$ is positive definite. 

ii)See   \cite[Lemma 4.2]{stoll-cremona}.

\qed

Choosing  $(t_1, \dots, t_n)$ that   make $\theta_0$ minimal  gives a unique positive definite quadratic $Q_f (X, Z)$.  We call this unique quadratic $Q_f (X, Z)$ for such a choice of $(t_1, \dots , t_n)$ the \textbf{Julia quadratic} of $f(X, Z)$ and denote it by $\J_f (X, Z)$.  From the previous remarks, this is well defined.

In the next example we show how these coefficients are picked in the case of binary cubics with reals roots.

\begin{exa}
Let $f(X)=aX^3+bX^2+cX+d$ be a binary cubic with three real roots $\a_1, \a_2, \a_3$. We pick $t_1, t_2, t_3$ as follows:
\[ t_1 = (\a_2 -\a_3)^2 , t_2 =(\a_3-\a_1)^2  , t_3 =(\a_1-\a_2)^2 \]
and Julia quadratic is as follows
\[\J_f(X, Z)=   (\a_2 -\a_3)^2(X-\a_1)^2 + (\a_3-\a_1)^2 (X-\a_2)^2+ (\a_1-\a_2)^2 (X-\a_3)^2\]
We can express the Julia quadratic covariant in terms of the coefficient of $f(X)$ as follows
\[\J_f(X, Z)=   (b^2-3ac)X^2 + (bc-9ad)X+ (c^2-3bd)\]
up to a constant factor.  
\end{exa}

The proof of the following lemma can be found in  \cite{julia-reduction}.  

\begin{lem} 
i)  $\theta_0 $ is an $SL_2(\R)$ invariant of the binary  form $f(X,Z)$.  

ii)  $\J_f (X,Z)$ is an $SL_2(\R)$ covariant of $f(X, Z)$
\end{lem}

In the literature $\theta_0 (f)$ is known as \textbf{Julia's invariant} of the binary form $f(X,Z)$. Julia gave explicitly $(t_1, \dots, t_n)$ only for cubics and quartics, and then Stoll and Cremona, in \cite{stoll-cremona}, provide a method for determining $t_i$'s (and therefore both $\theta_0 $ and  $\J_f (X,Z)$)  for binary forms of degree $n \geq 2$, as described in the next subsection.

Next, we make the necessary adjustments such that the above construction will work for binary forms with complex coefficients as well.

\subsection{Reduced binary forms with complex coefficients }\label{binary-red}

 Let $B_n$ be the space of degree $n$ binary forms in $K[X, Z]$, where $K$ is either $\R$ or $\C$,   and $f(X, Z)$ a stable binary form  in $B_n$ given as follows 
\[ f(X, Z) =a_0X^n+a_1X^{n-1}Z+ \cdots +a_nZ^n   \]
and suppose that $a_0 \neq 0$. Then, $f(X, Z)$ can be factored as 
\begin{equation}\label{roots}
 f(X, Z) =a_0 (X-\a_1Z)(X-\a_2Z) \cdots (X-\a_nZ), 
\end{equation}
for $\a_i \in \C$.   Construct a  positive definite quadratic form 
\[Q(X,Z) =  \sum_{j=1}^n t_j \cdot (X-\a_j Z) (X- \bar{\a}_i Z)     =\sum_{j=1}^n t_j|X-\a_jZ|^2, \] 
where $t_j$ are positive real numbers that have to be determined.  The following is true. 

\begin{lem} 
i) $Q_f (X,Z)$ is a positive definite quadratic Hermitian form

ii) There exists a unique tuple    $\mathfrak t_f= (t_1, \dots , t_n)$ which  make $\theta_0 (Q_f)$
\[ \theta_0 (Q_f) = \frac{|a_0|^2 (\D( Q_f))^{n/2}}{n^n \, t_1 \cdots t_n}\]
minimal.
\end{lem}

The proof is analogue to the proof of  Lemma~\ref{julian-quad-min}.
%
%
%
%
%
%
As previously, it can be proved that: 
i)  $\theta_0 $ is an $SL_2(K)$ invariant of the binary  form $f(X,Z)$, and     ii)  $\J_f (X,Z)$ is an $SL_2(K)$ covariant of $f(X, Z)$.

Thus, to  each binary form $f$ of degree $n>2$ we associate a unique  positive definite binary quadratic Hermitian form  $\J_f (X,Z)$. 
Next, we will show how to extend the zero map of quadratic forms to the set of  degree $n$ binary forms.

\def\zz{\bar \z}

We define the zero map for a binary form as 
\[ 
\begin{split}
\zz : & B_n \longrightarrow \H_3 \\
&f  \to \z (\J_f ) \\
\end{split}
\]
where $\z$ is as defined in ~\eqref{zero-map-hermitian} if $K=\C$, and as defined in ~\eqref{zero-map-real} if $K=\R$.  Note that  $\z (\J_f )$ is the point in $\H_2$ ($\H_3$) associated to the binary quadratic (respectively Hermitian) form $\J_f$.   We proved in \ref{equivariant-real} (respectively \ref{equivariant-hermitian}) that the zero map is an $SL_2(\Z)$-equivariant map ($SL_2(\C)$-equivariant map),  therefore for any $M\in SL_2 (\Z)$ ($M\in SL_2 (\C)$)  the  following holds for any binary form $f$ 
\[ \zz(f^M) = M^{-1} \zz(f). \]

Now we can define the  binary form to be reduced in analogy to part 1 and part 2. 


\begin{defn}
A stable binary form $f(X,Z)\in \R[X, Z]$ is said to be a \textbf{reduced binary form} if and only if $\zz(f) \in \F$, where $\F$ is the fundamental domain of of $SL_2(\Z)$.  A complex degree $n$ binary  forms $f(X, Z)$ is reduced if $\zz(f)$ is in a fixed fundamental domain for the actin of $SL_2(\O_K)$.
\end{defn}

For real forms, the covariance of $\zz(f)$ implies that each $SL_2(\Z)$-orbit of stable real binary forms contains at least one reduced form $f$.  Usually there will be exactly one reduced form in each orbit unless $\zz (f)$ is on the boundary of the fundamental domain, when there may be two. 

To find the reduced form first we compute $\zz(f)$ then if $\z(f)\in \F$ we are done. Otherwise,  find an $M \in SL_2(\Z)$ such that $M^{-1} \cdot \zz(f) \in \F$ as explained in 
Theorem~\eqref{equivariant-hermitian}. Then,   
\[ f^{M^{-1}} = f(dX-bZ, -cX+aZ) \]
is the reduced form of $f(X, Z)$.

Another property of Julia's  invariant,  as shown below, is that Julia's invariant bounds the leading term of the binary form, as well as the roots.  

\begin{lem}[Julia]
If $\z (f) \in \F$, then $a_0 \leq \frac{1}{3^{n/2}n^n} \cdot \theta_0$. 
\end{lem}

Julia, also showed that one can bound the magnitude of the roots $|\a_i|$  of $f(X, Y)$ in terms of $\theta_0/a_0$, i.e.
\[ |\a_i|^2  \leq \frac{1}{(n-1)^{n-1} \, 3^{n/2}} \cdot \frac{ \theta_0}{a_0^2}.\]
For more details about bounds see \cite{baker}, \cite{cremona-red}, and \cite{bhargava-yang}.

\section{An algorithm for reduction of binary forms}

In this section we   describe briefly an algorithm of Cremona and Stoll as in \cite{stoll-cremona} for computing the Julia quadratic and then the reduction of the binary form.   Unfortunately, the algorithm is based on computing the roots of the binary invariant.

%
%
%

Let $f(X, Z)$ be a binary form of degree $n$ written as in Eq.~\eqref{roots}.  To determine the $t_1, \dots , t_n$ coefficients of the Julia quadratic we solve for $t$ and $u$ the following system
\begin{equation}\label{system-zero}
\left\{
\begin{split}
& \sum_{j=1}^n \frac{u^2}{|t-\a_j|^2+u^2} = \frac n 2\\
&  \sum_{j=1}^n \frac{t-\a_j}{|t-\a_j|^2+u^2} = 0
\end{split}
\right.
\end{equation}
Then, the coefficients $t_1, \dots , t_n$ are given by
\begin{equation}\label{t_i}
t_i= \frac  2 n \frac{su^2}{|t-\a_j|^2+u^2}, \quad \textit{ for } i=1, \dots , n.
\end{equation}
and without loos of generality we can assume $s=1$. In \cite{stoll-cremona} it is proved that for a stable form $f\in B_n$, the representative point $\bar \z(f) \in \H_2$ (or $\H_3$)  is given as $\bar \z(f) = (t, u)$, where $(t,u)$ is the unique solution in $K \times R^+$ of the system \eqref{system-zero}.

Every solution to the system gives rise to a critical point in a compact domain $D$, which then must be the unique minimizing point of the Julia invariant.  Hence, we can compute $\bar \z(f)$ numerically, by performing a search for solutions of the above system.

\subsection{Implementation issues}
This is a summary of Section 6 in  \cite{stoll-cremona}. Let $F(X, Z)$ be a  stable binary form with degree $n\geq 3$ and coefficients in $\R$ and let $f(X)= F(X, 1)$. Define 
\[ Q_{0, F}(X, Z) = \sum_{i=1}^n \frac{(X-\a_iZ)(X-\bar \a_i Z)}{|f^\prime(\a_i)|^{2/(n-2)} }. \]
In \cite{stoll-cremona} it is proved that for all $n \geq 3$, $Q_{0, F}$ is positive definite and a covariant of $F$.  Denote with $\z_0 = \z (Q_{0, F})$.  If $F$ is a real form with distinct roots then $Q_{0, F}(X, Z)$ is well defined and as we saw in Part 1 has  a root $\z_0 $  in $ \H_2$, see equation~\eqref{zero-map-real}.  A binary form $F$ is sad to be \textbf{ $Q_0$-reduced} if $\z_0(F) \in  \F$, where $\F$ is the fundamental domain of $SL_2(\Z)$ acting on $\H_2$, as described in part 1.   When $\z_0(F) \in  \F$, $\z(F)$ is expected to be not very far from $\F$ and therefore we can bring in $\F$ with a couple more moves. 

Note that using this definition, $Q_0$-reduced, instead of the usual one given  in Section~\ref{binary-red}  is more convenient since $Q_{0, F}$  is easily written down. But note that this does not give optimal results if $F$ is a  binary forms with degree $n \geq 5$ as shown in 
\cite[Section 6]{stoll-cremona}. The algorithm  to reduce a binary form $F(X,Z)$ is as follows.

Firstly, we  compute $\z_0(F)$ numerically.  While $\z_0(F)$ is outside $\F$ repeat the following:

  i)  To get $\re(\z_0(F))$ inside the strip $-\frac 1 2$, $\frac 1 2$, let $m$ be the nearest integer to $\re(\z_0(F))$.  Then, let $\z_0(F)= \z_0(F) - m$ and we perform the inverse operation on $F(X, Z)$, i.e $F(X, Z)= F(X+mZ, Z)$.

  ii)  We want $|\z_0(F)| > 1$, if not let $\z_0(F) = - \frac{1}{\z_0(F)} $ and set $F(X, Z) = F(Z, -X)$. 

At the end we compute $\z(\J_F)$ where $\J_F$ is Julia's quadratic. If this is not in $\F$ we perform the same operations as described above and the forms $F$ will be the reduced form.

Now we are ready to summarize the  algorithm as follows. 

\begin{alg} Reduction Algorithm \\

\noindent \textbf{Input:} A stable degree $n>2$ binary  form $F(X,Z) \in \R[X, Z] $.  

\noindent \textbf{Output:}   A reduced binary form $G(X, Z)$ in the $SL_2(\Z)$-orbit of $F(X,Z)$. \\

\noindent \textbf{Step 1:}  Compute
\[ Q_{0, F}(X, Z) = \sum_{i=1}^n \frac    1     {|f^\prime(\a_i)|^{2/(n-2)} }  \,  {(X-\a_iZ)(X-\bar \a_i Z)}. \]
where $\a_i$ are the roots of $F(X,Z)$.  \\


\noindent  \textbf{Step 2:}   Compute the zero map image $\zz_0 (F)  = \z (Q_0)$. \\ 

\noindent \noindent  \textbf{Step 3:}      While $\zz (F)= \re(\zz (F)) + i \im(\zz (F))$ is outside $\F$ repeat the following.\\
 
  \textbf{shift}  Let $m$ be the nearest integer to $\re(\zz_0 (F))$.  Then, let $\zz_0(F)= \zz_0(F) - m$ and we perform the inverse operation on $G(X, Z)$, i.e $F(X, Z)= F(X+mZ, Z)$. \\

   \textbf{invert}    We want $|\zz_0(F)| > 1$, if not let $\zz_0(F) = - \frac{1}{\zz_0(F)} $ and set $G(X, Z) = F(Z, -X)$.  \\

\noindent  \textbf{Step 5:}  Compute Julia's quadratic covariant $\J_F$. \\

\noindent  \textbf{Step 6:} Compute the zero $\zz (F)=\z (\J_F)$. If $\zz( F) \in \F$, we are done,  otherwise repeat Step 3  for $\zz( F)$. \\
 
\end{alg}

The reduction  algorithm is implemented in Magma by Stoll and Cremona, and Sage by Streng and Bouyer, see \cite{streng} for details.  Since we are after the curve with minimal absolute height we applied Stoll-Cremona algorithm implemented in Magma to several curves  to check if this algorithm gives (or not) the curve with minimal height.  

We did the following computations. Start with a genus two  hyperelliptic curve  with height 1.  Compute it's Igusa-Clebsch invariants and then recover the hyperelliptic curve using this invariants. Reduce this curve using Stoll-Cremona reduction algorithm in Magma. 

First we did  this computations for 40 curves with automorphism group $\Z_2$ and in all  cases we got  a twist of the hyperelliptic  curve with height 1 that we started with.  

Then, we did the same  computations for  genus two hyperelliptic curves with automorphism group $D_8$.   We found six cases (out of ten)  where the reduced curve by Stoll-Cremona was not a twist of the original curve with height 1, see the following example.                                                          

\begin{exa}
Let $C$ be the genus 2 curve given by
\[ y^2=  4x^5+4x^3-3x\]

Over $\C$ this curve is isomorphic to the curve  $C_{216}$ in Table 1 of \cite{nato-8}, which has equation
\[ y^2=  -x -x^2+x^4+x^5 .\]
Therefore, this curve $C$ has minimal absolute height 1.

By Cremona-Stoll algorithm implemented in Magma, the minimal model of this curve is 
\[ y^2=3x^5-4x^3-4x\]
which has height 4. This curve is not a twist of the curve $C_{216}$ in \cite{nato-8}.  
\end{exa}

In the following table we give the other five cases. In the first column we give the genus two curve $C$ that we want to reduce, in the second column we give the reduced curve that we get from Magma, and then in the third column we give the curve with height 1 which is $GL_2(\C)$-isomorphic with  $C$. 

\begin{small}
\begin{table}[h]  \label{tab-1}
\centering
\begin{tabular}{|c| c| c|} 
\hline
$g= 2$ curve  & Stoll-Cremona reduced  curve  & Curve with height 1   \\
\hline
$18x^5+18x^3+x$ & $x^5+18x^3+ 18x $ & $-1-x -x^2-x^4+x^5 -x^6$\\
 \hline
$ 100x^5+100x^3+13x$ & $13x^5+100x^3+100x$ & $-1-x -x^5 + x^6 $\\
\hline
$20x^5+20x^3+x$ & $-x^5-20x^3 -20x$ & $ -1-x +x^5 + x^6$ \\
\hline
$16x^5+ 16x^3+5x$ & $-2x^5-8x^3 -10x$ & $  -1-x +x^2+x^4+x^5 - x^6$\\
\hline
$ 36x^5+ 36x^3+x$ & $  x^5+36x^3 +36x$ & $-1 -x^2-x^4- x^6$\\
\hline
\end{tabular} 
\vspace{.3cm}
\end{table}
\end{small}

\section{Further remarks}

The goal of these lectures was to give a survey on the reduction of binary forms and its recent developments.  The work of Cremona and Stoll builds on classical works of Julia and others and provides an efficient way to reduce binary forms up to $GL_2 (\Z)$-equivalence. However, their computation of the Julia quadratic is based on numerical techniques. 
 The purely algebraic approach would be to determine the coefficients of the Julia quadratic directly from the coefficients of the binary form.  It remains to be investigated if this can be achieved.  

In \cite{bin} we prove that the method of reduction via Julia quadratic gives indeed a form of minimal height in the corresponding $GL_2 (\Z)$-orbit.  However, as shown by our computations above it does not give a binary form with minimal absolute height in the sense of \cite{nato-8}.  In \cite{bin} we intend to give a complete treatment of how this can be achieved.


\nocite{*}
\bibliographystyle{amsplain} 

\bibliography{reduction-2}{}


\end{document}